%% file: goed_quad_paper_arxiv.tex
\documentclass[onecolumn]{svjour3}                     %
\smartqed  %
\input{macros.tex}

\graphicspath{{figs/}}

\numberwithin{equation}{section}
\begin{document}

\title{
Goal oriented optimal design of infinite-dimensional 
Bayesian inverse problems using quadratic approximations\\
}

\titlerunning{Goal-oriented OED}        %

\def\addressncsu{Department of Mathematics, North Carolina State University,
  Raleigh, NC, USA}
\def\addresssandia{Sandia National Laboratories,
  Albuquerque, NM, USA}

\author{
	J.~Nicholas Neuberger %
	\and
	Alen Alexanderian %
	\and
	Bart van Bloemen Waanders %
}

\institute{J.N.~Neuberger \at
              Department of Mathematics,
              North Carolina State University, \\
              Raleigh, NC. 
           \and
           A.~Alexanderian \at
              Department of Mathematics,
              North Carolina State University, \\
              Raleigh, NC.\\
           \and
		   B.v.B.~Waanders\\
		   Center for Computing Research,
		   Sandia National Labs,\\
  		   Albuquerque, NM.
}

\maketitle

\begin{abstract} 
We consider goal-oriented optimal design of experiments for infinite-dimensional 
Bayesian linear inverse problems governed by partial differential equations (PDEs). 
Specifically, we seek sensor
placements that minimize the posterior variance of a prediction or goal quantity
of interest. The goal quantity is assumed to be a nonlinear functional of the
inversion parameter. We propose a goal-oriented optimal experimental design (OED) 
approach that uses a
quadratic approximation of the goal-functional to define a goal-oriented
design criterion.  The proposed criterion, which we call the
$\gq$-optimality criterion, is obtained by integrating the posterior variance of the quadratic
approximation over the set of likely data.  Under the assumption of Gaussian
prior and noise models, we derive a closed-form expression for this criterion.
To guide development of discretization invariant computational methods, the
derivations are performed in an infinite-dimensional Hilbert space setting.
Subsequently, we propose efficient and accurate computational methods for
computing the $\gq$-optimality criterion.  A greedy approach is used to obtain
$\gq$-optimal sensor placements.  We illustrate the proposed approach for two
model inverse problems governed by PDEs. Our numerical results demonstrate the
effectiveness of the proposed strategy. In particular, the proposed approach 
outperforms non-goal-oriented (A-optimal) and linearization-based (c-optimal)
approaches. 
\end{abstract}

\section{Introduction}
Inverse problems are common in science and engineering applications. In such
problems, we use a model and data to infer uncertain parameters, henceforth 
called inversion parameters, that are not directly observable.  
We consider the
case where measurement data are collected at a set of sensors.  In practice,
often only a few sensors can be deployed. Thus, optimal placement of the sensors
is critical.  Addressing this requires solving an optimal experimental design
(OED) problem~\cite{AtkinsonDonev92,Ucinski05,Pukelsheim06}.

In some applications, the estimation of the inversion parameter is merely an 
intermediate step.  For example, consider a source inversion problem in a heat
transfer application. In such problems, one is often interested in prediction
quantities such as the magnitude of the temperature within a region of interest 
or heat flux through an interface.  A more complex example is a wildfire simulation
problem, where one may seek to estimate the source of the fire, but the emphasis
is on prediction quantities summarizing future states of the system.  In such
problems, design of experiments should take the prediction/goal quantities 
of interest into account.  Failing to do
so might result in sensor placements that do not result in optimal uncertainty
reduction in the prediction/goal quantities.  This points to
the need for a \emph{goal-oriented} OED approach. This is the subject of this
article.  

We focus on Bayesian linear inverse problems
governed by PDEs with infinite-dimensional parameters. 
To make matters concrete, we consider the observation model, 
\begin{equation}\label{eq:observation_model} 
	\vec y = \iF m + \vec \eta.
\end{equation}
Here, $\vec y \in \R^d$ is a vector of measurement data, $\iF$ is a linear
parameter-to-observable map, $m$ is the inversion parameter, and $\vec\eta$
is a random variable that 
models measurement noise.  We consider the case where $m$ belongs to an
infinite-dimensional real separable Hilbert space $\ms{M}$ and $\iF:\ms{M}\to
\R^d$ is a continuous linear transformation.  The inverse problem seeks to
estimate $m$ using the observation model~\eqref{eq:observation_model}. Examples
of such problems include source inversion or initial state estimation in linear
PDEs. See Section~\ref{sec:background}, for a brief summary of the requisite background
regarding infinite-dimensional Bayesian linear inverse problems and OED for such
problems. 

We consider the case where solving the inverse problem is an intermediate 
step and the primary focus is accurate estimation of a scalar-valued prediction 
quantity characterized by a nonlinear goal-functional,
\begin{equation}\label{eq:goal_operator}
	\Q:\ms{M} \to \R.
\end{equation}
In the present work, we propose a 
goal-oriented OED approach that seeks to find sensor placements 
minimizing the posterior uncertainty in such goal-functionals.

\boldheading{Related work}
The literature devoted to OED is extensive. Here, we discuss articles that are closely 
related to the present work.  
OED for infinite-dimensional Bayesian linear inverse problems has been addressed
in several works in the past decade; see 
e.g.,~\cite{AlexanderianPetraStadlerEtAl14,AlexanderianSaibaba18,HermanAlexanderianSaibaba20}.
Goal-oriented approaches for OED in inverse problems governed by differential equations 
have appeared in~\cite{HerzogRiedelUcinski18,Li19,ButlerJakemanWildey20}. 
The article~\cite{HerzogRiedelUcinski18} considers nonlinear problems with nonlinear 
goal operators. In that article, a goal-oriented OED criterion is obtained using 
linearization of the goal operator and 
an approximate (linearization-based) covariance matrix for the inversion parameter.
The thesis~\cite{Li19} considers linear inverse problems with Gaussian prior and noise models, 
where the goal operator itself is a linear transformation of the inversion parameters. 
A major focus of that thesis is the study of methods for the combinatorial optimization problem
corresponding to optimal sensor placement. 
The work~\cite{ButlerJakemanWildey20} considers a stochastic inverse problem 
formulation, known as data-consistent framework~\cite{ButlerJakemanWildey18}. This approach, 
while related, is different from traditional Bayesian inversion. 
Goal-oriented OED for infinite-dimensional linear inverse problems was studied
in~\cite{AttiaAlexanderianSaibaba18,WuChenGhattas23a}. These articles consider
goal-oriented OED for the case of linear parameter-to-goal mappings. 

For the specific class of problems considered in the present work, a traditional
approach is to consider a linearization of the goal-functional $\Q$ around a
nominal parameter $\bar{m}$. Considering the posterior variance of this
linearized functional leads to a specific form of the well-known c-optimality
criterion~\cite{ChalonerVerdinelli95}.  However, a linear approximation does not
always provide sufficient accuracy in characterizing the uncertainty in the
goal-functional.  In such cases, a more accurate approximation to $\Q$ is
desirable. 

\boldheading{Our approach and contributions}
We consider a quadratic approximation of the goal-functional. Thus, $\Q$ is approximated by
\begin{equation}\label{eq:quadratic_approx}
	\Q(m) \approx \Q_\qua(m) \defeq \Q(\bar{m}) + \ip{\nabla \Q(\bar m), m-\bar{m}} 
	   + \frac12 \ip{\nabla^2 \Q(\bar m)(m - \bar m), m - \bar m}.
\end{equation}
Following an A-optimal design approach, we consider the posterior variance of
the quadratic approximation, $\mathbb{V}_{\postm} \{\Q_\qua\}$.  We derive an
analytic expression for this variance in the infinite-dimensional setting, in
Section~\ref{sec:gOED}.  Note, however, that this variance expression depends on
data $\vec y$, which is not available a priori.  To overcome this, we compute
the expectation of this variance expression with respect to data. This results in 
a data-averaged design criterion, which we call
the $\gq$-optimality criterion.  Here, $G$ indicates the goal-oriented nature of
the criterion and $q$ indicates the use of a quadratic approximation.  The
closed-form analytic expression for this criterion is derived in
Theorem~\ref{theorem:main}. 

Subsequently, in Section~\ref{sec:method}, we present three computational
approaches for fast estimation of $\Psi$, relying on Monte Carlo trace
estimators, low-rank spectral decompositions, or a low-rank singular value
decomposition (SVD) of $\iF$, respectively. Focusing on problems where the goal
functional $\Q$ is defined in terms of PDEs, our methods rely on adjoint-based
expressions for the gradient and Hessian of $\Q$. We demonstrate effectiveness
of the proposed goal-oriented approach in a series of computational experiments
in Section~\ref{sec:numerics_quad} and Section~\ref{sec:numerics_nonlinear}. The
example in Section~\ref{sec:numerics_quad} involves inversion of a volume source
term in an elliptic PDE with the goal defined as a quadratic functional of the
state variable.  The example in Section~\ref{sec:numerics_nonlinear} concerns a
porous medium flow problem with a nonlinear goal functional. 

The key contributions of this article are as follows: 
\begin{itemize}
\renewcommand{\labelitemi}{$\bullet$}
\item derivation of a novel goal-oriented
design criterion, the $\gq$-optimality criterion, based on a quadratic approximation of the goal-functional, in
an infinite-dimensional Hilbert space setting (see Section~\ref{sec:gOED}); 
\item efficient computational
methods for estimation of the $\gq$-optimality criterion (see Section~\ref{sec:method});
\item extensive computational experiments, demonstrating the importance of
goal-oriented OED and effectiveness of the proposed approach (see Section~\ref{sec:results}).
\end{itemize}

\section{Background}\label{sec:background}
In this section, we discuss the requisite background concepts 
and notations regarding Bayesian linear inverse problems and OED. 
\subsection{Bayesian linear inverse problems}\label{sec:inverse_problem}
The key components of a
Bayesian inverse problem are the prior distribution, the data-likelihood, and the
posterior distribution. The prior encodes our prior knowledge about the inversion parameter, 
which we denote by $m$. 
The likelihood, which incorporates the parameter-to-observable map, 
describes the conditional distribution of data for a given 
inversion parameter. Finally, the posterior is a distribution law for $m$ that 
is conditioned on the observed data and is consistent with the prior.  These
components are related via the Bayes formula~\cite{Stuart10}. 
Here, we summarize the process for the case of linear Bayesian 
inverse problem.

\boldheading{The data likelihood}
We consider a bounded linear parameter-to-observable map, $\iF:\ms{M} \to \R^d$.
In linear inverse problems governed by PDEs, we define $\iF$ as the composition of 
a linear PDE solution operator $\iS$ and a linear observation operator $\iB$,
which extracts solution values at a prespecified set of measurement points. 
Hence, $\iF = \iB \iS$.
In the present, work, we consider observation models of the form
\begin{equation}\label{eq:inc_mod}
	\vec y = \iF m + \vec \eta, \quad \text{where} \quad \vec \eta \sim \mf{N}(0,\sig^{2}\mat{I}).
\end{equation} 
We assume $m$ and $\vec \eta$
are independent, which implies, $\vec y | m \sim \mf{N}(\iF m,
\sig^{2}\mat{I})$. This defines the data-likelihood.

\boldheading{Prior}
Herein, $\ms{M} = L^2(\Omega)$, where $\Omega$ is a bounded domain in
two- or three-space dimensions. This space is equipped with the $L^2(\Omega)$ inner 
product $\ip{\cdot,\cdot}$ and norm $\| \cdot \| = \ip{\cdot,\cdot}^{1/2}$. 
We consider a Gaussian prior law
$\mu_{\pr} := \mf{N}(\mpr, \Cpr)$.
To define the prior, we follow the approach in~\cite{Stuart10,Bui-ThanhGhattasMartinEtAl13}.
The prior mean is assumed to be a sufficiently regular element of $\ms{M}$ 
and the prior covariance operator $\Cpr$ is defined as the inverse of a differential operator. 
Specifically, 
let $\mc{E}$ be the mapping $s \mapsto m$, defined by the solution operator of 
\begin{equation}\label{eq:elliptic_PDE}
	\begin{alignedat}{1}
		-a_{1}(a_{2}\Del m + m) &= s \quad \text{in } \Omega,\\
		\grad m \cdot \vec n  &= 0, \quad \text{on } \partial\Omega,
	\end{alignedat}
\end{equation}
where $a_{1}$ and $a_{2}$ are positive constants.  Then, the prior
covariance is defined as $\Cpr := \mc{E}^{2}$.  
\boldheading{Posterior}
For a Bayesian linear inverse problem with a Gaussian prior and a Gaussian noise model given by \eqref{eq:inc_mod}, 
it is well-known~\cite{Stuart10} that the posterior 
is the Gaussian measure $\postm \defeq \mf{N}\paren{\mMAPy, \Cpo}$ with 
\begin{equation}\label{eq:post}
\Cpo = \paren{\sig^{-2}\iF^{*}\iF + \Cpr^{-1}}^{-1} \quad \text{and} \quad \mMAPy = \Cpo\paren{\sig^{-2}\iF^{*}\vec y + \Cpr^{-1}\mpr},
\end{equation}
where $\iF^*$ denotes the adjoint of $\iF$.
Here, the posterior mean is the 
maximum a posteriori probability (MAP) point. 
Also, recall the variational 
characterization of this MAP point as the unique global minimizer of 
\begin{equation}\label{eq:J}
J(m) \defeq \frac{1}{2\sig^2} \|\iF m - \vec y\|^{2}_2 + \half \|m - \mpr\|_{\Cpr^{-1}}^{2}
\end{equation}
in the Cameron--Martin space, $\mathrm{range}(\Cpr^{1/2})$; see~\cite{DashtiStuart17}.
The Cameron--Martin space plays a key role in the study of Gaussian measures on Hilbert spaces.
In particular, this space is important in the theory of 
Bayesian inverse problems with Gaussian priors.
Here, 
$\| \cdot \|_{\Cpr^{-1}}^{2}$ is the Cameron--Martin norm, $\| m \|_{\Cpr^{-1}}^{2} = \| \Cpr^{-1/2} m \|^2$.

It can be shown that the Hessian of $J$, 
denoted by $\iH$, satisfies $\iH = \Cpo^{-1}$.
In what follows, the Hessian of data-misfit term in \eqref{eq:J} will be
important.  We denote this Hessian by $\iHmis \defeq \sig^{-2}\iF^{*}\iF$. 
A closely related operator is the prior-preconditioned data-misfit Hessian,
\begin{equation}\label{eq:prior_prec_misfitH}
\iHmist \defeq \Cpr^{1/2}\iHmis\Cpr^{1/2},
\end{equation}
which also plays a key role in the discussions that follow.

Lastly, we remark on the case when the forward operator is affine. 
This will be the case for inverse problems governed by linear 
PDEs with inhomogeneous source volume or boundary source terms.
The model inverse problem considered in
Section~\ref{sec:numerics_nonlinear} is an example of such problems.
In that case, the forward
operator may be represented as the affine map $\iG(m) = \iF m + \vec d$, where
$\iF$ is a bounded linear transformation.  Under the Gaussian assumption on the prior
and noise, the posterior is a Gaussian with the same covariance operator
as in~\eqref{eq:post} and with the mean given by  
$\mMAPy = \Cpo\paren{\sig^{-2}\iF^{*}(\vec y -d) + \Cpr^{-1}\mpr}$.

\boldheading{Discretization}
We discretize the inverse problem using the continuous Galerkin 
finite element method. 
Consider a nodal finite element basis of compactly supported 
functions $\{\phi_i\}_{i=1}^{N}$.
The discretized inversion parameter is represented as 
$m_h = \sum_{i=1}^N m_i \phi_i$. Following common practice, 
we identify $m_h$ with the vector of its finite element coefficients, 
$\vec{m} = [m_1 \; m_2 \; \cdots \; m_N]^\top$. 
The discretized inversion parameter space is thus $\R^N$ equipped 
with the mass-weighted inner product $\ip{\vec u, \vec v}_{\mat M} \defeq \vec u^\top \mat{M} \vec v$.
Here, $\mat{M}$ is the finite element mass matrix, 
$M_{ij} := \int_\Omega \phi_i, \phi_j \, d\vec{x}$, for $i, j \in \{1, \ldots, N\}$.
Note that this mass-weighted inner product is the discretized $L^2(\Omega)$ inner product.
Throughout the article, we use the notation $\mR^{N}_{\mat M}$ for 
$\mR^{N}$ equipped with the mass-weighted inner product $\ip{\cdot, \cdot}_{\mat M}$.

We use boldfaced symbols to represent the discretized versions of the operators
appearing in the Bayesian inverse problem formulation. For details on obtaining
such discretized operators, see~\cite{Bui-ThanhGhattasMartinEtAl13}.  
The discretized solution, observation, and forward operators are
denoted by $\mat S$, $\mat B$, and $\mat F$, respectively. Similarly, the
discretized Hessian is presented as $\mat H$. We denote the discretized prior
and posterior covariance operators by $\Gpr$ and $\Gpo$, respectively. Note 
that $\Gpr$ and $\Gpo$ are selfadjoint 
operators on $\mR^{N}_{\mat M}$. 
\subsection{Classical optimal experimental design}\label{sec:classical_OED}

In the present work, an experimental design corresponds to an array of sensors selected from
a set of candidate sensor locations, $\{x_{i}\}_{i=1}^{n} \subset \Omega$.  
In a \emph{classical} OED problem, 
an experimental design is called optimal if 
it minimizes a notion of posterior uncertainty in the inversion 
parameter. This is different from a goal-oriented approach, where 
we seek designs that minimize the uncertainty in a goal
quantity of interest.

To formulate an OED problem, it is helpful to parameterize sensor placements
in some manner.  A common approach is to assign weights to each sensor in the
candidate sensor grid.  That is, we assign a weight $w_i \geq 0$ to each $x_i$,
$i \in \{1, \ldots, n\}$.  This way, a sensor placement is identified with a
vector $\vec w \in \mR^n$.  Each $w_{i}$ may be restricted to some subset of
$\mR$ depending on the optimization scheme. Here, we assume
$w_{i} \in \{0, 1\}$; a weight of zero means the corresponding sensor is
inactive. 

The vector $\vec w$ of the design weights is incorporated in the Bayesian 
inverse problem formulation through the data-likelihood~\cite{Alexanderian21}.
This yields a $\vec w$-dependent posterior measure. In particular, 
the posterior covariance operator is given by 
\begin{equation}\label{eq:Cpo_weights}
	\Cpo(\vec w) = \paren{\iF^{*}\Ws\iF + \Cpr^{-1}}^{-1} \quad \text{with} \quad 
	\Ws = 			\sigma^{-2}\text{diag}(\vec w).
\end{equation}
There are several classical criteria in the OED literature. One example is  
the \emph{A-optimality} criterion, which is defined as the trace of $\Cpo(\vec w)$.
The corresponding discretized criterion is
\begin{equation}\label{eq:discretized_classical_criterion}
	\mat{\Theta}(\vec w) = \tr \paren{\Gpo(\vec w)}
	\quad \text{with} \quad \Gpo(\vec w) \defeq
	\paren{\fF^{*}\Ws\fF + \Gpr^{-1}}^{-1}.
\end{equation}
The A-optimality criterion quantifies the average posterior variance 
of the inversion parameter field. To define a goal-oriented analogue 
of the A-optimality criterion, we need to consider the posterior variance 
of the goal-functional $\Q$ in~\eqref{eq:goal_operator}. This is discussed 
in the next section. 

\section{Goal-oriented OED}\label{sec:gOED}
In a goal-oriented OED problem,
we seek designs that minimize 
the uncertainty in a goal quantity of interest, 
which is a function of the inversion parameter $m$. 
Here, we 
consider a nonlinear goal-functional 
\begin{equation}\label{eq:goal_fucntional}
	\Q: \ms{M} \to \mR.
\end{equation}
In our target applications, evaluating $\Q$ involves solving PDEs.  Thus,
computing the posterior variance of $\Q$ via sampling can be 
challenging---a potentially large number of samples might be required.  
Also, computing an optimal design requires evaluation of
the design criterion in every step of an optimization algorithm.  Furthermore,
generating samples from the posterior requires forward and adjoint PDE
solves. Thus, design criteria that require sampling $\Q$ at every step of an
optimization algorithm will be computationally inefficient. 
One approach to 
developing a computationally tractable goal-oriented OED 
approach is to replace $\Q$ by a suitable approximation. This 
leads to the definition of approximate measures of uncertainty in $\Q$.

We can use local approximations of $\Q$ to derive goal-oriented criteria.
This requires an
\emph{expansion point}, which we denote as $\mbar \in \ms{M}$. 
A simple choice is
to let $\mbar$ be the prior mean.  
Another approach, which might be feasible
in some applications, is to assume some initial measurement data
is available. This data may be used to compute an initial parameter estimate.
Such an initial estimate might not be suitable for prediction purposes, but can
be used in place of $\mbar$.  The matter of expansion point
selection is discussed later in Section~\ref{sec:results}. For now,
$\mbar$ is considered to be a fixed element in $\ms{M}$. 
In what follows, we assume $\Q$ is twice differentiable and denote
\begin{equation}\label{eq:goal_derivs}
\dQ := \grad \Q(\mbar) \quad \text{and} \quad
\ddQ := \grad^{2}\Q(\mbar).
\end{equation}

A known approach for obtaining an approximate measure of uncertainty in $\Q(m)$
is to consider a linearization of $\Q$ and compute the posterior variance of
this linearization. In the present work, this is referred to as the
$\gl$-optimality criterion, denoted by $\Psi^{\ell}$. The $G$ is used to
indicate goal, and $\ell$ is a reference to linearization. As seen shortly, 
this $\gl$-optimality criterion is a specific instance of the Bayesian
c-optimality criterion~\cite{ChalonerVerdinelli95}. Consider the linear
approximation of $\Q$ given by
\begin{equation}\label{eq:lin}
\Q(m) \approx \Q_{\text{lin}}(m) \defeq \Q(\mbar) + \ip{\dQ, m - \mbar}.
\end{equation}
The $\gl$-optimality criterion is 
\begin{equation}
\label{eq:psi_lin}
\Psi^{\ell}\defeq \mV_{\mu_{\po}}\crbra{\Q_{\text{lin}}}. 
\end{equation}  
It is straightforward to note that 
\begin{equation}\label{eq:coptimal_derivation}
\Psi^{\ell} = \mV_{\mu_{\po}}\crbra{\ip{\dQ, m - \mbar}} = 
\int_\ms{M} \ip{\dQ, m - \mbar}^2 \, \mu_{\po}(dm) = \ip{\Cpo \dQ, \dQ},
\end{equation}
where we have used the definition of covariance operator; see~\eqref{equ:covariance_definition}. 
Letting $c = \dQ$, we obtain the 
c-optimality criterion $\ip{ \Cpo c, c}$. 
The variance of the linearized goal is an intuitive and tractable choice 
for a goal-oriented criterion.  
However, a linearization might severely underestimate the posterior uncertainty in $\Q$
or be overly sensitive to choice of $\mbar$. 

In the present work, 
we define an OED criterion based on the quadratic Taylor expansion of $\Q$. 
This leads to the $\gq$-optimality criterion mentioned in the introduction.
Consider the quadratic approximation,
\begin{equation}\label{eq:quad}
	\Q(m) \approx \Q_{\qua}(m) \defeq \Q(\mbar) + \ip{\dQ, m - \mbar} + \half\ip{\ddQ(m - \mbar), m-\mbar}.
\end{equation}
We can compute  $\mV_{\mu_{\po}}\crbra{\Q_{\qua}}$ analytically.  
This is facilitated by
Theorem~\ref{theorem:variance} below. 
The result is well-known in the finite-dimensional setting.  In the
infinite-dimensional setting, this can be obtained from properties of Gaussian
measures on Hilbert spaces, some developments in~\cite{DaPratoZabczyk02} (cf.\
Remark 1.2.9., in particular), along with the formula for the expected value of
a quadratic form on a Hilbert space~\cite{AlexanderianGhattasEtAl16}. This
approach was used in~\cite{AlexanderianPetraStadlerEtAl17} to derive the
expression for the variance of a second order Taylor expansion, within the 
context of optimization under uncertainty.  However, to our knowledge a direct
and standalone proof of Theorem~\ref{theorem:variance}, which is of independent and
of broader interest, does not seem to be available in the literature. Thus, we present a
detailed proof of this result in the appendix for completeness.
\begin{theorem}[Variance of a quadratic functional]
\label{theorem:variance} 
Let $\iA$ be a bounded selfadjoint linear operator on a
Hilbert space $\ms{M}$ and let $b \in \ms{M}$. 
Consider the quadratic functional $\Q:\ms{M}\to\mR$
given by 
\begin{equation}\label{equ:quadratic_functional_def}
	\Q(m) \defeq \frac{1}{2}\ip{\iA m, m} + \ip{b, m},
	\quad m \in \ms{M}.
\end{equation} 
Let $\mu = \mf{N}(m_{0}, \iC)$ be a Gaussian measure on $\ms{M}$.
Then, we have 
\[
\mV_{\mu}\crbra{\Q} = \|\iA m_{0} + b\|_{\iC}^{2}
+\half\tr\big((\iC\iA)^{2}\big).  
\]
\end{theorem} 
\begin{proof}
See Appendix~\ref{appdx:variance_of_quad}.
\qed
\end{proof}
We next consider the posterior variance 
$\mV_{\postm}\crbra{\Q_{\qua}}$ of $\Q_{\qua}$.  
Using Theorem~\ref{theorem:variance}, we obtain
\begin{equation}\label{eq:variance_MAP}
	\mV_{\postm}\crbra{\Q_{\qua}} = 
	\left\|\ddQ\mMAPy + b\right\|_{\Cpo}^{2} + \half\tr\big((\Cpo\ddQ)^{2}\big), \quad \text{where} \, b = \dQ - \ddQ\mbar.	
\end{equation}
Note that this variance expression depends on 
data $\vec{y}$, which is  
not available when solving the OED problem. Indeed, the main 
point of solving an OED problem is to determine how data should be collected. 
Hence, we consider the ``data-averaged'' criterion,
\begin{equation}\label{eq:psi}
	\Psi := \mE_{\mu_{\pr}}\Big\{ \mE_{\vec y | m} \big\{ \mV_{\postm}\{\Q_{\qua}\}\big\}\Big\}.
\end{equation}
Here, $\mE_{\mu_{\pr}}$ and
$\mE_{\vec y | m}$ represent expectations with respect to the prior and
likelihood, respectively. This uses the information available in the Bayesian inverse problem formulation 
to compute the expected value of $\mV_{\mu_{\po}}\crbra{\Q_{\qua}(m)}$ over the set of likely data.
In the general case of nonlinear inverse problems, such averaged criteria 
are computed via sample averaging~\cite{AlexanderianPetraStadlerEtAl16,Alexanderian21}. However,  
in the present setting, exploiting the linearity of the parameter-to-observable map
and the Gaussian assumption on prior and noise models, we can compute 
$\Psi$ analytically. This is the main result of this section and presented in the following theorem.

\begin{theorem}[Goal-oriented criterion]
\label{theorem:main}
Let $\Psi$ be as defined in \eqref{eq:psi}. Then, 
\begin{equation}\label{eq:psi_full}
	\Psi = \|\ddQ(\mpr-\mbar) + \dQ\|_{\Cpo}^{2}
	+ \tr\paren{\Cpr\ddQ\Cpo\ddQ} -\half\tr\big((\Cpo\ddQ)^{2}\big).
\end{equation}
\end{theorem}
\begin{proof}
See Appendix~\ref{sec:proof_main_thm}.
\qed
\end{proof}
We call $\Psi$ in~\eqref{eq:psi_full} the 
\emph{$\gq$-optimality} criterion.
Proving Theorem~\ref{theorem:main} involves three main steps.  In the first
step, the variance of the quadratic approximation of $\Q$ is calculated using
Theorem~\ref{theorem:variance}.  This results in \eqref{eq:variance_MAP}.  In
the second step, we need to compute the nested expectations in~\eqref{eq:psi}.
Calculating these moments requires obtaining the expectations of linear and
quadratic forms with respect to the data-likelihood and prior laws.  The
derivations rely on facts about measures on Hilbert spaces.  Subsequently, using
properties of traces of Hilbert space operators, the definitions of the
constructs in the inverse problem formulation, and some detailed manipulations,
we derive~\eqref{eq:psi_full}.  See Appendix~\ref{sec:proof_main_thm}, for details.

\section{Computational Methods}
\label{sec:method}
Computing the $\gq$-optimality criterion~\eqref{eq:psi_full} requires computing traces of
high-dimensional and expensive to apply operators, which is a computational
challenge. To establish a flexible computational framework, in this section, we
present three different algorithms for fast estimation of the $\gq$-optimality
criterion. 
In Section~\ref{sec:goed_mc}, we
present an approach based on randomized trace estimators. Then, we present an
algorithm that uses the low-rank spectral decomposition of the
prior-preconditioned data misfit Hessian in Section~\ref{sec:goed_spectral}.
Finally, in Section~\ref{sec:goed_lowrankF}, we present an approach that uses
the low-rank SVD of the prior-preconditioned forward operator.  In each 
case, we rely on structure exploiting methods to obtain scalable
algorithms. 

Before presenting these methods, we briefly discuss the discretization of the
$\gq$-optimality criterion. In addition to the discretized operators presented
in Section~\ref{sec:inverse_problem}, we require access to the discretized goal
functional, denoted as $\q$, and its derivatives. In what follows, we denote 
\begin{equation}\label{eq:goal_derivs_finite}
	\dq \defeq \grad \q(\vec \mbar), \quad \ddq \defeq \grad^{2}\q(\vec \mbar), \quad \text{and} \quad \bq \defeq \ddq(\fmpr-\fmbar) + \dq.
\end{equation} 
The discretized $\gq$-optimality criterion is given by 
\begin{equation}\label{eq:goed_discretized}
    \mat\Psi(\vec w) = \ipM{\Gpo(\vec w) \bq, \bq}
	+ \tr\!\paren{\Gpr\ddq\Gpo(\vec w)\ddq} - \half\tr\big((\Gpo(\vec w)\ddq)^{2}\big). 
\end{equation}
Similarly, discretizing the $\gl$-optimality criterion $\Psi^{\ell}$, presented in \eqref{eq:psi_lin}, yields
\begin{equation}\label{eq:psi_lin_disc}
\mat \Psi^{\ell}(\vec w) = \ipM{\Gpo(\vec w) \dq, \dq}.
\end{equation}

\subsection{A randomized algorithm}
\label{sec:goed_mc}
In large-scale inverse problems, it is expensive to build the
forward operator, the prior and posterior covariance operators, or the
Hessian of the goal-functional, $\ddq$ in \eqref{eq:goal_derivs_finite}. 
Therefore, matrix-free methods that only require applications of these 
operators on vectors are essential. 
A key challenge here is computation of the traces in 
\eqref{eq:goed_discretized}.
In this section, we present an approach for computing $\mat\Psi$ that relies on 
randomized trace estimation~\cite{Avron2011}. 
As noted in~\cite{AlexanderianPetraStadlerEtAl14}, 
the trace of a linear operator $\mat T$ on 
$\mR^{N}_{\mat M}$ can be approximated via
\begin{equation}\label{eq:rand_trace}
	\tr\paren{\mat T} \approx \frac{1}{p}\sum_{j=1}^{p} \ipM{\mat T \vec \xi_{j}, \vec \xi_{j}}, \quad \text{with} \ \vec \xi_{j} \sim \mf{N}(\vec{0}, \mat M^{-1}).
\end{equation}
This is known as a Monte Carlo trace estimator. 
The number $p$ of the required trace estimator vectors is problem-dependent. 
However, 
in practice, often a modest $p$ (in order 
of tens) is sufficient for the purpose of optimization. 

We use Monte Carlo estimators to approximate the trace terms in \eqref{eq:goed_discretized}. 
In particular, we use 
\begin{multline*}
\tr\paren{\Gpr\ddq\Gpo\ddq} - \half\tr\big((\Gpo\ddq)^{2}\big)
= 
\tr\big( (\Gpr - \half\Gpo) \ddq\Gpo\ddq\big) \\ 
\approx 
\frac{1}{p}\sum_{j=1}^{p}\ipM{ (\Gpr - \half\Gpo) \ddq\Gpo\ddq \vec\xi_i,\vec\xi_i} 
= 
\frac{1}{p}\sum_{j=1}^{p}\ipM{(\Gpr - \half\Gpo)\vec \xi_{i},\ddq\Gpo\ddq\vec \xi_{i}},
\end{multline*} 
where we have exploited the fact that $\Gpr$ and $\Gpo$ are selfadjoint with 
respect to the mass-weighted inner product $\ipM{\cdot,\cdot}$. 
Thus, letting 
\[
T_p \defeq \frac{1}{p}\sum_{j=1}^{p}\ipM{(\Gpr - \half\Gpo)\vec \xi_{i},\ddq\Gpo\ddq\vec \xi_{i}}, 
\]
we can estimate $\mat\Psi$ by
\begin{equation}\label{eq:PsiRand}
\mat\Psi \approx \PsiRand \defeq \ipM{\Gpo \bq, \bq} + T_p.
\end{equation}
This enables a computationally tractable approach for approximating
$\mat\Psi$.  We outline the procedure for computing $\mat\PsiRand$ in
Algorithm~\ref{alg:randomized}. The computational cost of this approach is
discussed in Section~\ref{sec:computational_cost}. 

The utility of methods based on Monte Carlo trace estimators in the context of OED 
for large-scale inverse problems has been demonstrated in previous studies such 
as~\cite{HaberHoreshTenorio08,HaberMagnantLuceroEtAl12,AlexanderianPetraStadlerEtAl14}.
A key advantage of the present approach is its simplicity.  However, further
accuracy and efficiency can be attained by exploiting
the low-rank structures embedded in the inverse problem.
This is discussed in the next section.

\begin{algorithm}[ht]
	\caption{Algorithm for estimating $\mat \PsiRand$.}
	\begin{algorithmic}[1]
	\STATE \textbf{Input:} $\{\vec \xi_{i}\}_{j=1}^{p}$, $\Gpo$
	\STATE \textbf{Output:} $\mat \PsiRand$
	\STATE Compute $\bq = \ddq(\vec \mpr-\vec \mbar) + \dq$
	\STATE Compute $\vec s = \Gpo \bq$
	\STATE Set $T = 0$
	\FOR{$j=1$ to $p$}
		\STATE Compute $\vec t_{1} = (\ddq \Gpo \ddq)\vec \xi_{j}$
		\STATE Compute $\vec t_{2} = (\Gpr - \half \Gpo)\vec \xi_{j}$
		\STATE Set $T = T + \ipM{\vec t_{1}, \vec t_{2}}$
	\ENDFOR
	\STATE Compute
	$\mat \PsiRand = \ipM{\vec s, \bq} + T/p$
	\end{algorithmic}
	\label{alg:randomized}
\end{algorithm}

\subsection{Algorithm based on low-rank spectral decomposition of $\fHmist$}
\label{sec:goed_spectral}

Here we present a structure-aware algorithm
for estimating the $\gq$-optimality criterion that exploits low-rank components 
within the inverse problem. Namely, we leverage
the often present low-rank structure in the (discretized)
prior-preconditioned data
misfit Hessian, $\fHmist \defeq \sig^{-2}\Gpr^{1/2}\fF^{*}\fF\Gpr^{1/2}$.

Let us denote
\[
	\ddqt   \defeq \Gpr^{1/2} \ddq \Gpr^{1/2} \quad\text{and}\quad
	\st     \defeq (\fHmist + \mat I)^{-1}.
\]
Note that the posterior covariance operator can be represented as 
\begin{equation}\label{eq:postrior_with_P}
\Gpo = \Gpr^{1/2} \st \Gpr^{1/2}.
\end{equation}

As shown in~\cite{Bui-ThanhGhattasMartinEtAl13}, we can obtain a computationally
tractable approximation of $\Gpo$ using a low-rank representation of $\fHmist$.
Let $\{(\lambda_i, \vec v_i)\}_{i=1}^k$ be the dominant eigenpairs of $\fHmist$. We use
\[
\fHmist \approx \mat V_{k}\mat\Lambda_{r}\mat V_{k}^{*} = \sum_{i=1}^{k}\lam_{i} \vec v_{i} \otimes \vec v_{i},
\]
where $\mat V_{k} = [\vec{v}_1 \; \vec{v}_2 \; \cdots \vec{v}_k]$ 
and $\mat \Lambda_{k} = \text{diag}(\lambda_1, \ldots, \lambda_k)$. 
Now, define $\gam_{i} \defeq
\lam_{i}/(\lam_{i}+1)$ and $\mat D_{k} =\text{diag}(\gam_{1}, \gam_{2}, \cdots,
\gam_{k})$.  We can approximate $\st$ using the Sherman-Woodbury-Morrison
formula, 
\begin{equation}\label{eq:St_lowrank}
\st \approx \rst \defeq \mat I - \mat V_{k}\mat D_{k}\mat V_{k}^{*} = \mat I - \sum_{i=1}^{k}\gam_{i}\vec v_{i} \otimes \vec v_{i}.
\end{equation}
Substituting $\st$ by $\rst$ in~\eqref{eq:postrior_with_P}, yields the approximation 
\begin{equation}\label{eq:Cpo_r}
\Gpo \approx \rGpo \defeq \Gpr^{1/2}\rst\Gpr^{1/2} = \Gpr - \Gpr^{1/2}\mat V_{k}\mat D_{k}\mat V_{k}^{*}\Gpr^{1/2}.
\end{equation}
Subsequently, the $\gq$-optimality criterion \eqref{eq:goed_discretized} 
is approximated by  
\begin{equation}\label{eq:rPsi}
	\mat\Psi_{k} \defeq
	\ipM{\Gpo \bq, \bq} + \tr\paren{\Gpr\ddq\rGpo\ddq} -\half\tr\big((\rGpo\ddq)^{2}\big). 
\end{equation}
The following result provides a convenient expression for computing 
$\mat\Psi_{k}$.
\begin{proposition}\label{prp:rPsi}
Let $\mat\Psi_{k}$ be as in~\eqref{eq:rPsi}. Then, 
\begin{equation}\label{eq:rPsi_expression}
\mat\Psi_{k} = \ipM{\rGpo \bq, \bq} + \frac12  \tr(\ddqt^2)
   - \frac12  \sum_{i,j=1}^k \gamma_i \gamma_j \ipM{\ddqt \vec v_{i}, \vec v_j}^2.
\end{equation}
\end{proposition}
\begin{proof}
See Appendix~\ref{appdx:spectral}.
\qed
\end{proof}

Note that the second term in \eqref{eq:rPsi_expression}, $\tr(\ddqt^2)$, is a constant that 
does not depend on the experimental design (sensor placement). 
Therefore, when seeking to optimize $\mat \Psi_k$ as a function of $\vec w$, we can neglect that constant term 
and focus instead on minimizing the functional
\begin{equation}\label{eq:psi_r_reduced}
	\PsiSpec \defeq \ipM{\rGpo \bq, \bq}
	- \frac12  \sum_{i,j=1}^k \gamma_i \gamma_j \ip{\ddqt \vec v_{i}, \vec v_j}_{\mat{M}}^2.
\end{equation}
The spectral approach for estimating the $\gq$-optimality criterion is outlined
in Algorithm~\ref{alg:spectral}. 
\begin{algorithm}[ht]
	\caption{Algorithm for computing $\PsiSpec$.}
	\begin{algorithmic}[1]
		\STATE \textbf{Input:} method for applying $\fHmist$ to vectors
		\STATE \textbf{Output:} $\PsiSpec$
		\STATE Compute the leading eigenpairs $\{(\lambda_i, \vec v_i)\}_{i=1}^k$ of $\fHmist$
		\STATE Set $\gamma_i = \lambda_i / (1 + \lambda_i)$, $i = 1, \ldots, k$
        \STATE Compute $\tilde{\vec{v}}_i = \Gpr^{1/2}\vec{v}_i$, for $i = 1, \ldots, k$ 
		\STATE Compute $\tilde{\vec q}_i = \ddq \tilde{\vec{v}}_i$ 
		\STATE Compute $\bq = \ddq(\vec \mpr-\vec \mbar) + \dq$ 
		\STATE Compute $\vec s = \Gpr \bq - \sum_{i=1}^k \gamma_i \ipM{\vec b, \tilde{\vec{v}}_i} \tilde{\vec{v}}_i$
		\hfill\COMMENT{$\vec s = \rGpo \bq$}
		\STATE Compute 
		\[
			\PsiSpec = \ipM{\vec s, \bq}  -\frac12  \sum_{i,j=1}^k 
	            \gamma_i \gamma_j \ipM{\tilde{\vec q}_i, \tilde{\vec{v}}_j}^2
		\]
	\end{algorithmic}
	\label{alg:spectral}
\end{algorithm}

Note that the approximate 
posterior covariance operator $\rGpo$ can be used to 
estimate the classical A-optimality criterion as well. Namely, we can use 
$\tr(\Gpo) \approx \tr(\rGpo)
= \tr(\Gpr) - \tr\paren{\Gpr^{1/2}\mat V_{r}\mat D_{k}\mat V_{k}^{*}\Gpr^{1/2}}$.
Since $\Gpr$ is independent of the experimental design, A-optimal designs can be obtained by minimizing
\begin{equation}\label{eq:theta_r_reduced}
	\mat\Theta_\text{k} \defeq - \tr\paren{\Gpr^{1/2}\mat V_{k}\mat D_{k}\mat V_{k}^{*}\Gpr^{1/2}}.
\end{equation}
Furthermore, the present spectral approach can also be used for fast computation of 
the $\gl$-optimality criterion. 
In particular, it is straightforward to note 
that $\gl$-optimal designs can be computed by minimizing
\[
\PsiSpec^{\ell} \defeq -\sum_{i=1}^k \gamma_i \ipM{\vec{v}_i, \dq}^2.	
\]
This is accomplished by substituting $\rGpo$ into the discretized 
$\gl$-optimality criterion, given by \eqref{eq:psi_lin_disc}, and performing some basic  
manipulations.

\subsection{An approach based on low-rank SVD of $\fF$}
\label{sec:goed_lowrankF}
In this section, we present an algorithm for estimating the $\gq$-optimality
criterion that relies on computing a low-rank SVD of
prior-preconditioned forward operator.
This approach relies on the specific form of the $\vec{w}$-dependent 
posterior covariance operator; see~\eqref{eq:discretized_classical_criterion}.

Before outlining our approach, we make the additional definitions
\begin{equation}\label{eq:low_rank_defs}
\fFt  \defeq \fF\Gpr^{1/2}, \quad \fFtw \defeq \Ws^{1/2} \fFt,\quad \fDw  \defeq (\mat{I} + \fFtw \fFtw^*)^{-1}, \quad \fPtw  \defeq (\mat{I} + \fFt^*\Ws\fFt)^{-1}.
\end{equation}

The following result enables a tractable representation for the $\gq$-optimality criterion.
\begin{proposition}\label{proposition:low_rank}
Consider the operators as defined in \eqref{eq:low_rank_defs}. 
The following hold:

\begin{tabular}{ll}
\quad(a) & $\fPtw = \mat{I} - \fFtw^* \fDw \fFtw$;\\
\quad(b) & $\tr\paren{\Gpr\ddq\Gpo\ddq} = \tr(\ddqt^2) -\tr(\fFtw \ddqt^2\fFtw^*\fDw)$;\\
\quad(c) & $\tr\big((\Gpo\ddq)^{2}\big) = 
\tr(\ddqt^2)-2\tr(\fFtw \ddqt^2\fFtw^*\fDw) + \tr\Big( (\fDw\fFtw \ddqt\fFtw^*)^2\Big)$. 
\end{tabular}
\end{proposition}
\begin{proof}
See Appendix~\ref{sec:proof_of_prp_lowrank}.
\qed
\end{proof}
Using Proposition~\ref{proposition:low_rank}, we can 
state the $\gq$-optimality criterion $\mat\Psi$ in~\eqref{eq:goed_discretized}
as 
\begin{equation}\label{eq:goed_discretized_alt}
    \mat\Psi(\vec w) = \ipM{\Gpo(\vec w)\bq, \bq}
    + \frac12 \tr(\ddqt^2) -\frac12 \tr\big( (\fDw\fFtw \ddqt\fFtw^*)^2\big).
\end{equation}
Note that the second term is independent of the design weights $\vec{w}$. Thus, 
we can ignore this term when minimizing $\mat\Psi$. In that case, we 
focus on
\begin{equation}\label{eq:goed_discretized_Phi}
    \PsiSVD \defeq  \ipM{\Gpo(\vec w)\bq, \bq} - \frac12 \tr\big( (\fDw\fFtw \ddqt\fFtw^*)^2\big).
\end{equation}
Computing the first term requires applications of $\Gpo$ to vectors. 
We note,
\[
\Gpo(\vec w) \vec{v} = \Gpr^{1/2} \fPtw \Gpr^{1/2} \vec v = 
\Gpr^{1/2} 
(\mat{I} - \fFtw^* \fDw \fFtw)\Gpr^{1/2} \vec v, \qquad 
\vec v \in \R^N.
\]
This only requires a linear solve in the measurement space,
when computing $\fDw \fFtw\Gpr^{1/2} \vec{v}$. 
Once a low-rank 
SVD of $\fFt$ is available, 
this can be done without performing any PDE solves.
The trace term in \eqref{eq:goed_discretized_Phi}
can also be computed efficiently.  First, we build 
\[
\mat{Q} \defeq \fFtw\ddqt\fFtw^*   
\]
at the cost of $d$ applications of $\ddq$ to vectors. 
The remaining part of computing $\hat{\mat\Psi}$ does not require any PDE solves. 
Let $\ip{\cdot,\cdot}_{2}$ denote the Euclidean inner product and $\{ \vec{e}_i\}_{i=1}^d$ 
the standard basis in $\R^d$.
We have
\begin{equation}\label{eq:trace_term_Phi}
    \tr\big( (\fDw\fFtw \ddqt\fFtw^*)^2\big)
    = \sum_{i=1}^d \ip{\fDw\mat{Q}\vec{e}_i, \mat{Q}\fDw\vec{e}_i}_{2}
\end{equation}
Computing this expression requires calculating $\fDw \vec{e}_i$, 
for $i \in \{1, \ldots, d\}$, which amounts to building $\fDw$. 
We are now in a position to present and an algorithm for computing 
the $\gq$-optimality criterion using a low-rank SVD of $\fFt$. This is summarized in Algorithm~\ref{alg:low_rank}.
\begin{algorithm}[ht]
	\caption{Algorithm for computing $\PsiSVD$}.
	\begin{algorithmic}[1]
		\STATE \textbf{Input:}
	    Precomputed $\vec{s} = \Gpr^{1/2}(\ddq(\vec \mpr-\vec \mbar) + \dq)$ and 
		\STATE
		\phantom{Input:~} rank-$r$ approximation $\fFtwr\approx\tilde\fF$ 
		\hfill\COMMENT{only applications of $\fFtwr$ to vectors are required}
		\STATE \textbf{Output:} $\PsiSVD$
		\STATE Build $\fDw = \big(\mat I + \fFtwr(\fFtwr)^{*}\big)^{-1}$
		\STATE Build $\mat Q = \fFtwr\ddqt(\fFtwr)^*$
		\STATE Compute 
		\[
			\PsiSVD = \ipM{\big[\mat{I} - (\fFtwr)^* \fDw \fFtwr\big]\vec s, \vec s} -\half \sum_{i=1}^{d}\ip{\fDw \mat Q\vec e_{i}, \mat Q \fDw \vec e_{i}}_{2}
		\]
	\end{algorithmic}
	\label{alg:low_rank}
\end{algorithm}

\subsection{Computational cost}\label{sec:computational_cost}
Here, we discuss the computational cost of the three algorithms presented above.
We measure complexity in terms of applications of the operators $\fF$ and
$\fF^{*}$, $\Gpr$, and $\ddq$. Note that $\fF$ and $\fF^*$
correspond to forward and adjoint PDE solves.
First we highlight the key computational considerations for each algorithm. 
\begin{description}
\item[Computing $\PsiRand$:] 
The bottleneck in evaluating $\PsiRand$ is the need for $p$ applications of
$\Gpo$ and $2p$ applications of $\ddq$.  We assume that a Krylov iterative
method is used to apply $\Gpo$ to vectors, requiring  $\mc{O}(r)$ iterations. In
the present setting, $r$ is determined by the numerical rank of the
prior-preconditioned data-misfit Hessian.  Thus, each application of $\Gpo$ 
requires $\mc{O}(r)$ forward and adjoint solves. 

\item[Computing $\PsiSpec$:]
In Algorithm~\ref{alg:spectral}, we need to compute the $k$ leading eigenpairs of $\fHmist$.
In our implementation, we use the Lanczos method, costing $\mc{O}(k)$
applications of $\fF$ and $\fF^{*}$. 
Note also that Algorithm~\ref{alg:spectral} requires $k + 1$ applications of $\ddq$ to vectors.

\item[Computing $\PsiSVD$:]
This algorithm requires a low-rank SVD of $\fFt$ computed up-front.  This can be
done using a Krylov iterative method or randomized
SVD~\cite{HalkoMartinssonTropp11}.  In this case, a rank $r$ approximation costs
$\mc{O}(r)$ applications of $\fF$ and $\fF^{*}$. This algorithm also requires 
$d$ applications of $\ddq$. 
\end{description}

For readers' convenience, we summarize the computational complexity of 
the methods in Table~\ref{tab:complexity}.
\begin{table}[ht]
	\centering
	\caption{Computational complexity of the three algorithms in Section~\ref{sec:method} 
	in terms of number of forward/adjoint solves ($\fF$/$\fF^{*}$) and goal-Hessian applications 
	($\ddq$). Randomized, spectral, 
	and low-rank SVD refer to Algorithm~\ref{alg:randomized}, Algorithm~\ref{alg:spectral}, 
	and Algorithm~\ref{alg:low_rank}, respectively. Note that the integer $r$, 
	required in Algorithm~\ref{alg:randomized} and Algorithm~\ref{alg:low_rank}, is independently selected.}
	\begin{tabular}{lcc}
		\toprule
		Algorithm & $\fF$/$\fF^{*}$ & $\ddq$\\
		\midrule
		randomized & $N_{tr} \mc{O}(r)$ & $2 p + 1$\\
		spectral & $\mc{O}(k)$ & $k+1$\\
		low-rank SVD & - & $d$\\
		\bottomrule
	\end{tabular}
	\label{tab:complexity}
\end{table}

A few remarks are in order. Algorithm~\ref{alg:spectral} and
Algorithm~\ref{alg:low_rank} are more accurate than the randomized approach in
Algorithm~\ref{alg:randomized}. When deciding between the spectral and low-rank
SVD algorithms, several considerations must be accounted for. First, we note
that the spectral approach is particularly cheap when the size of the desired
design, i.e., the number of active sensors, is small. 
If the design vector $\vec w$ is such that
$\|\vec w\|_1 = k$, then $\text{rank}(\fHmist) \leq k$. Thus, we only require the
computation of the $k$ leading eigenvalues of $\fHmist$. The low-rank SVD
approach is advantageous in the case when the forward model $\fF$ is expensive
and applications of $\ddq$ are relatively cheap. This is due to the fact that no
forward or adjoint solves are required in Algorithm~\ref{alg:low_rank}, after precomputing
the low-rank SVD of $\tilde{\fF}$. However, the number of applications of $\ddq$
is fixed at $d$, where $d$ is the number of candidate sensor locations. Lastly,
we observe that all algorithms given in Section~\ref{sec:method} may be modified to
incorporate the low-rank approximation of $\ddq$. Implementation of this is
problem specific. Thus, methods presented are agnostic to the structure of
$\ddq$

\section{Computational experiments}\label{sec:results}

In this section we consider two numerical examples.  The first one, concerns
goal-oriented OED where the goal-functional is a quadratic functional.
In that case, the second order Taylor expansion provides an exact representation
of the goal-functional.  That example is used to provide an intuitive
illustration of the proposed strategy; see Section~\ref{sec:numerics_quad}.  In
the second example, discussed in Section~\ref{sec:numerics_nonlinear}, the
goal-functional is nonlinear. In that case, we consider the inversion of a
source term in a pressure equation, and the goal-functional is defined in terms
of the solution of a second PDE, modeling diffusion and transport of a
substance. That example enables testing different aspects of the proposed
framework and demonstrating its effectiveness.  In particular, we demonstrate the
superiority of the proposed $\gq$-optimality framework over the $\gl$-optimality and classical A-optimality
approaches, in terms of reducing uncertainty in the goal. 

\subsection{Model problem with a quadratic goal functional}
\label{sec:numerics_quad}
Below, we first describe the model inverse problem under study and the goal-functional.
Subsequently, we present our computational results. 

\subsubsection{Model and goal}
We consider the estimation of the source term $m$ 
in the following stationary advection-diffusion equation: 
\begin{equation}\label{eq:model1}
\begin{alignedat}{2}
-\alp \Delta u + \vec{v} \cdot \grad u &= m, \quad &\text{in} \ \Omega,\\
u &\equiv 0, \quad &\text{on} \ E_{1},\\
\grad u \cdot \vec{n} &\equiv 0, \quad &\text{on} \ E_{2}.
\end{alignedat}
\end{equation}
The goal-functional is defined as the $L^2$ norm of the solution to
\eqref{eq:model1}, restricted to a subdomain $\Omega^* \subset \Omega$.
To this end, we consider the restriction operator
\[
    (\mc{R}u)(\vec x) \defeq 
		\begin{cases}
			u(\vec x) 	\quad \text{if } \vec{x} \in \Omega^*,\\
			0			\quad \text{if} x \in \Omega \setminus \Omega^{*},
		\end{cases}
\]
and define the goal-functional by 
\[
	\Q(m) := \frac{1}{2}\ip{ \mc{R}u(m), \mc{R}u(m)}, \quad m \in \ms{M}.
\]
Recalling that $\iS$ is the solution operator to \eqref{eq:model1}, we can equivalently describe the goal as
\begin{equation}\label{eq:goal1}
\Q(m) = \frac{1}{2}\ip{\iA m, m }, \quad \text{where} \quad \iA := \iS^{*}\mc{R}^*\mc{R}\iS.
\end{equation}

\subsubsection{The inverse problem}
In~\eqref{eq:model1}, we take the diffusion constant to be $\alp = 0.1$ and
velocity as $\vec v = [0.1, -0.1]$.  Additionally, 
we let $E_1$ be the union of the left and top edges of $\Omega$ and 
$E_2$ the union of the right and bottom edges.
\begin{figure}[ht]
	\centering
		\includegraphics[width=.45\textwidth]{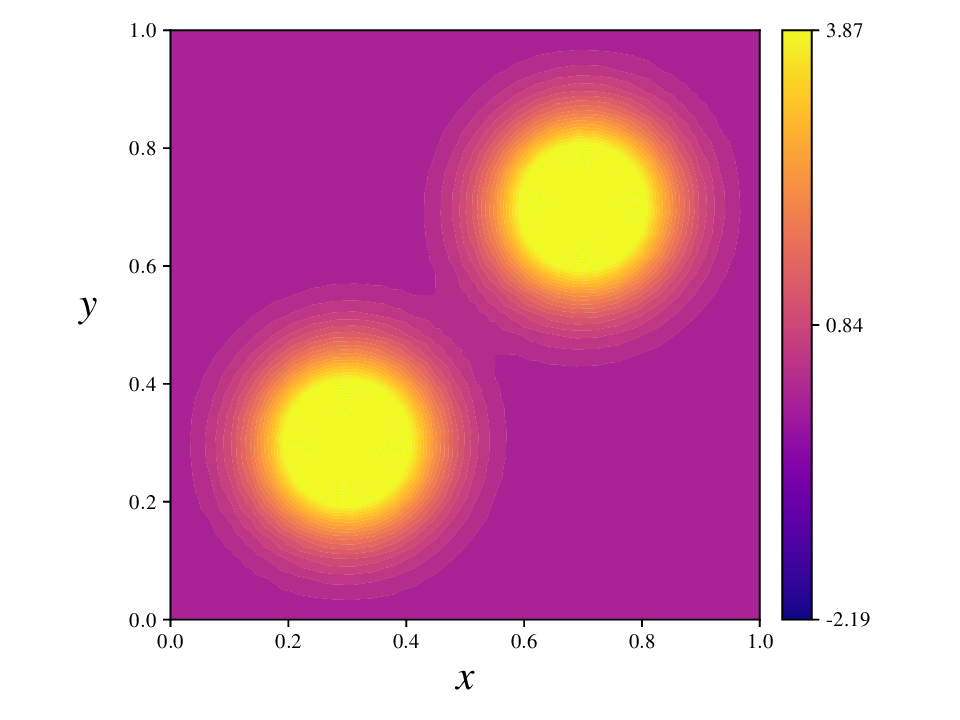}
		\includegraphics[width=.45\textwidth]{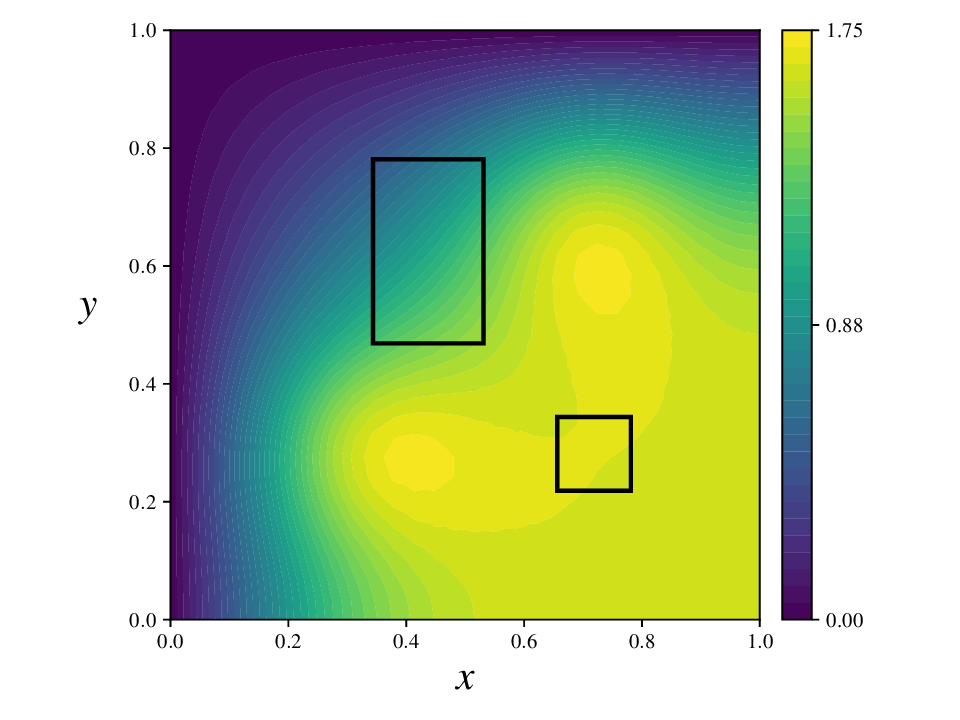}
	\caption{The true inversion parameter $\mtrue$ (left) and corresponding state solution $p(\mtrue)$ (right). The subdomain $\Omega^{*}$ (black rectangles) is depicted in the left figure.}
    \label{fig:model1}
\end{figure}
In the present experiment, we use a ground truth parameter $m_{\text{true}}$, 
defined as the sum of two Gaussian-like functions to generate a 
data vector $\vec y$.  We depict
our choice of $m_\text{true}$ and the corresponding state solution 
in Figure~\ref{fig:model1}. Note that in Figure~\ref{fig:model1}~(right), 
we also depict our choice of the subdomain $\Omega^*$ for the present example.
Additionally, the noise variance is set to $\sig^{2} =
10^{-4}$. This results in a roughly $\%1$ noise-level.  As for the prior, we
select the prior mean as $\mpr\equiv 4$ and use $(a_{1}, a_{2}) =
(8\cdot10^{-1}, 4^{-2})$ in~\eqref{eq:elliptic_PDE}.
As an illustration, we visualize the
MAP-point and several posterior samples in Figure~\ref{fig:quad_posterior}.

\begin{figure}[ht]
	\centering
	\includegraphics[width=\textwidth]{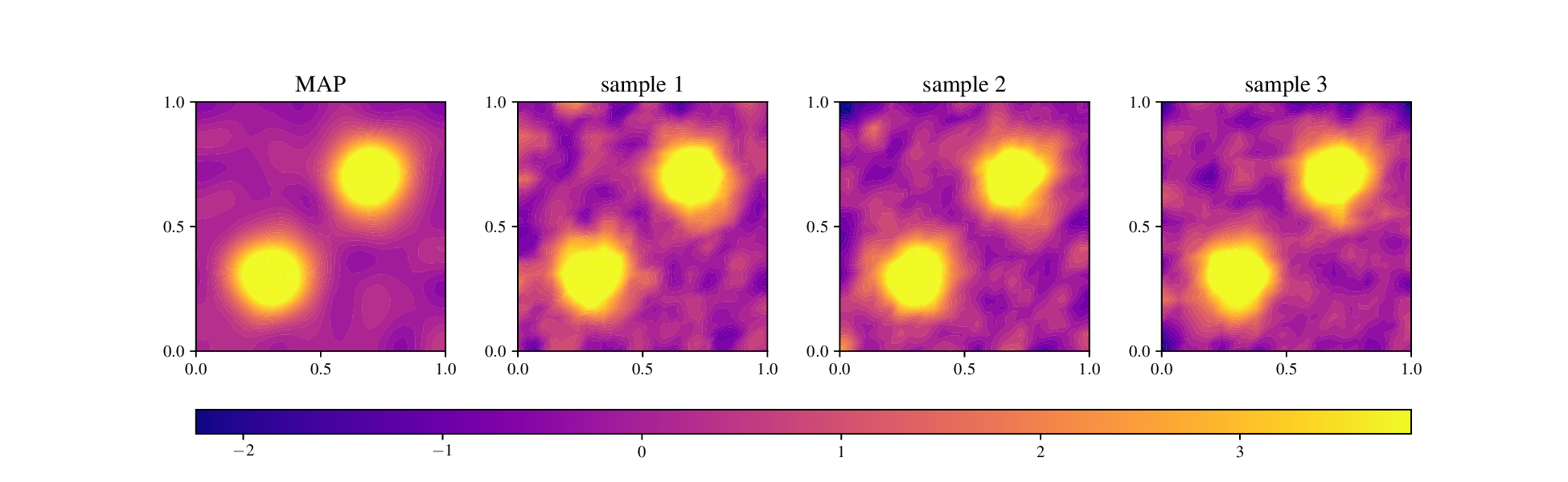}
	\caption{MAP point (leftmost) and three posterior samples. The posterior is obtained using 
	data collected from the entire set of $N_{s}$ candidate sensor locations.
	}
	\label{fig:quad_posterior}
\end{figure}

For all numerical experiments in this paper, we use a continuous Galerkin finite element discretization
with piecewise linear nodal basis functions and $N_{x} = 30^{2}$ spatial grid
points.  Regarding the experimental setup, we use $N_{s} = 15^{2}$ candidate
sensor locations distributed uniformly across the domain.
Implementations in the present work are conducted in python and 
finite element discretization is performed with FEniCS~\cite{fenics2015}.

\subsubsection{Optimal design and uncertainty}
In what follows, we choose the spectral method for computing the classical and
goal-oriented design criteria, due to its accuracy and computational efficiency. 
A-optimal designs are obtained by minimizing
\eqref{eq:theta_r_reduced}. As for the $\gq$-optimality criterion, we implement the spectral
algorithm as outlined in Algorithm~\ref{alg:spectral}.
Let $\mat A$ be the discretized version of operator $\iA$ in~\eqref{eq:goal1}. 
In the context of this problem, 
the $\gq$-optimality criterion, resulting from \eqref{eq:psi_r_reduced}, is
\begin{equation}\label{eq:quad_criterion}
	\PsiSpec(\vec w) = \ip{\rGpo \mat A \vec \mpr, \mat A \vec \mpr}_{\mat M}-\frac12  \sum_{i,j=1}^k 
	 \gamma_i \gamma_j \ipM{\mat A \vec v_{i}, \vec v_j}^2.
\end{equation}	            

\begin{figure}[ht]
	\centering
	\includegraphics[width=\textwidth]{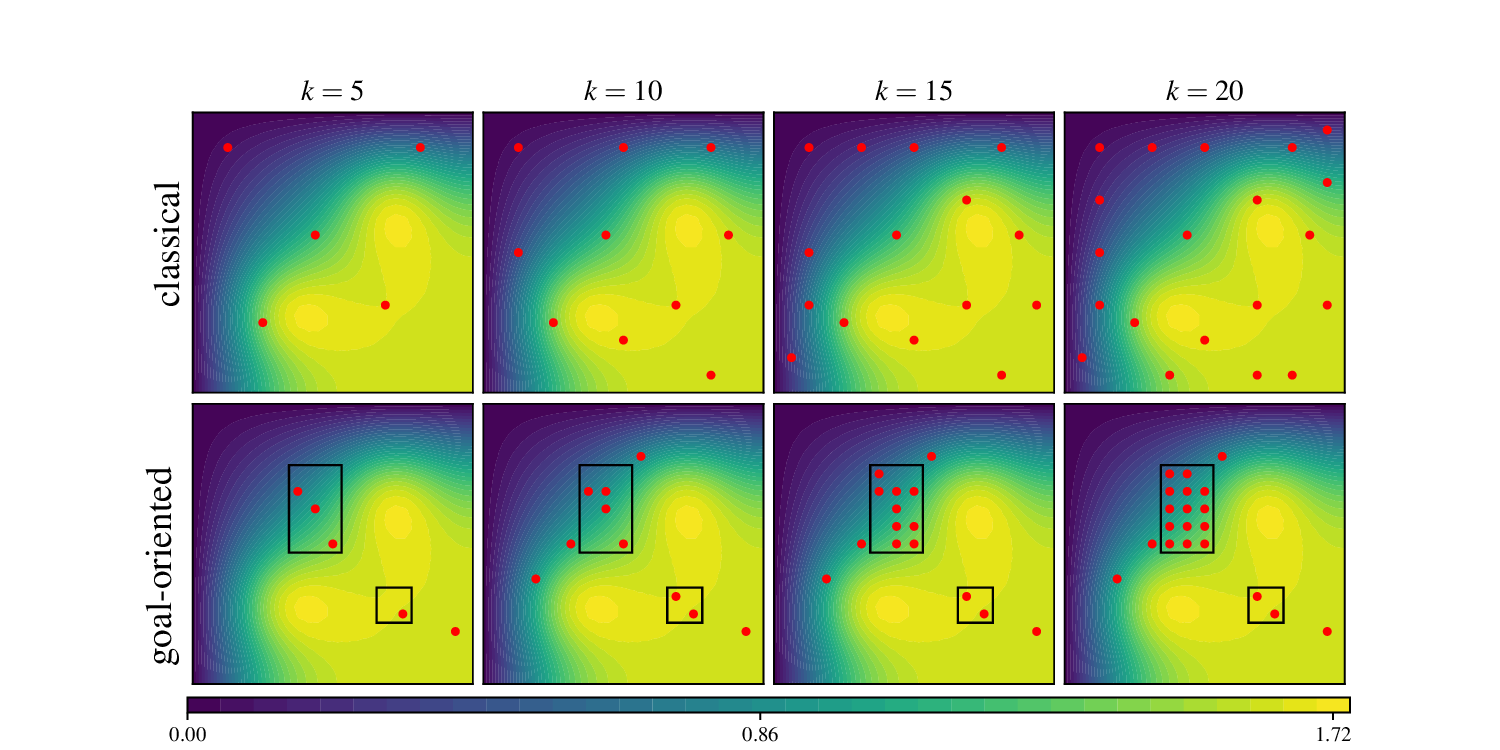}
	\caption{Classical (A-optimal) and goal-oriented ($\gq$-optimal) designs of size $k\in\{5, 10, 15, 20\}$ plotted over the true state solution. The subdomain $\Omega^{*}$ (black rectangles) are plotted over the goal-oriented plots.}
	\label{fig:classic_goal_designs1}
\end{figure}

Both classical A-optimal and goal-oriented $\gq$-optimal designs are obtained with the greedy algorithm.
As a first illustration, we plot both types of designs over the state solution
$u(\mtrue)$; see Figure~\ref{fig:classic_goal_designs1}.  Note that for the
$\gq$-optimal designs, we overlay the subdomain $\Omega^{*}$, used in the
definition of the goal-functional $\Q$ in~\eqref{eq:goal1}.
In Figure~\ref{fig:classic_goal_designs1}, we observe that the classical
designs tend to spread over the domain, while the $\gq$-optimal designs cluster
around the subdomain $\Omega^*$. 
However, 
while the goal-oriented sensor placements
prefer the subdomain, sensors are not exclusively placed within this region. 

We next illustrate the effectiveness of the $\gq$-optimal designs in reducing
the uncertainty in the goal-functional, as compared to A-optimal designs.  In
the left column of Figure~\ref{fig:compare_quad}, we consider posterior
uncertainty in the goal-functional (top) and the inversion parameter (bottom)
when using classical designs with $k = 5$ sensors. Uncertainty in the goal
functional is quantified by inverting on a given design, then propagating
posterior samples through the goal functional. We refer to the computed
probability density function of the goal values as a \emph{goal-density}.
Analogous results are reported in the right column, when using goal-oriented
designs. Here, the posterior distribution corresponding to each design is
obtained by solving the Bayesian inverse problem, where we use data synthesized
using the ground-truth parameter.  

We observe that the $\gq$-optimal designs are far more effective in reducing
posterior uncertainty in the goal-functional. The bottom row of
the figure reveals that the goal-oriented designs are more effective in reducing
posterior uncertainty in the inversion parameter in and around the subdomain
$\Omega^*$. On the other hand, the classical designs, being agnostic
to the goal functional, attempt to reduce uncertainty in the inversion parameter
across the domain.  While this is intuitive, we point out that the nature of the
goal-oriented sensor placements are not always obvious. Note that for the
$\gq$-optimal design reported in Figure~\ref{fig:compare_quad}~(bottom-right), a sensor is placed near the right boundary. This implies that
reducing uncertainty in the inversion parameter around this location is
important for reducing the uncertainty in the goal-functional. In general, sensor placements are influenced by physical parameters 
such as the velocity field, modeling assumptions 
such as boundary conditions, as well as the definition of the goal-functional.

\begin{figure}[ht]
	\centering
		\includegraphics[width=.75\textwidth]{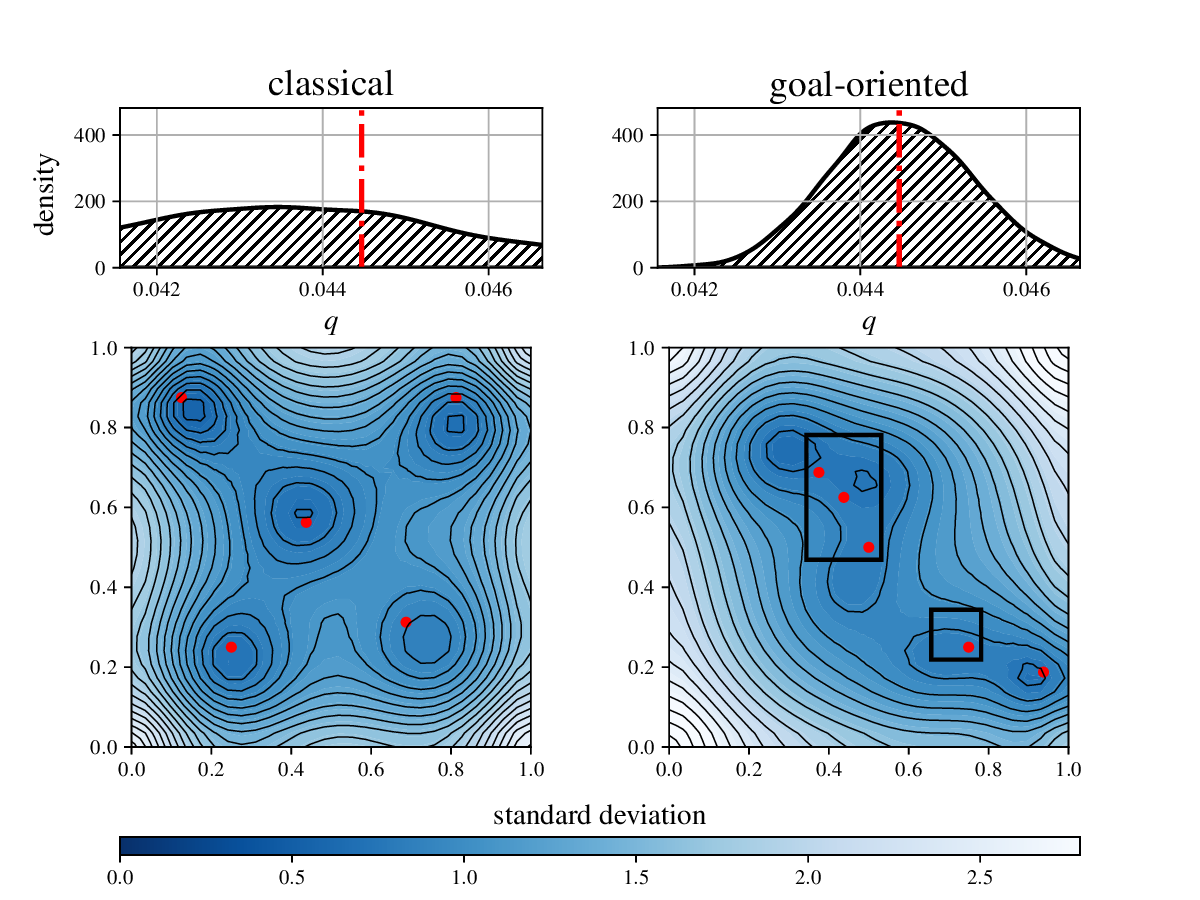}
	\caption{Bottom row: classical (A-optimal) and goal-oriented ($\gq$-optimal) designs of size $k=5$ plotted over the respective posterior standard deviation fields. 
	Top row: posterior goal-densities constructed by propagating posterior samples through $\q$. The dashed line is the true goal value.}
	\label{fig:compare_quad}
\end{figure}

To provide further insight, 
we next consider 
classical and goal-oriented designs with varying number of sensors. 
Specifically, 
we plot the corresponding goal-densities against each other 
in Figure~\ref{fig:violin1}, as the size $k$ 
of the designs increases.
We observe that the densities corresponding to the $\gq$-optimal designs 
have a smaller spread and are closer to the true goal value, when compared to
the densities obtained using the A-optimal designs. 
This provides further evidence that our goal-oriented OED framework is more
effective in reducing uncertainty in the goal-functional when compared to the classical design approach.
\begin{figure}[ht]
	\centering
	\includegraphics[width=1\textwidth]{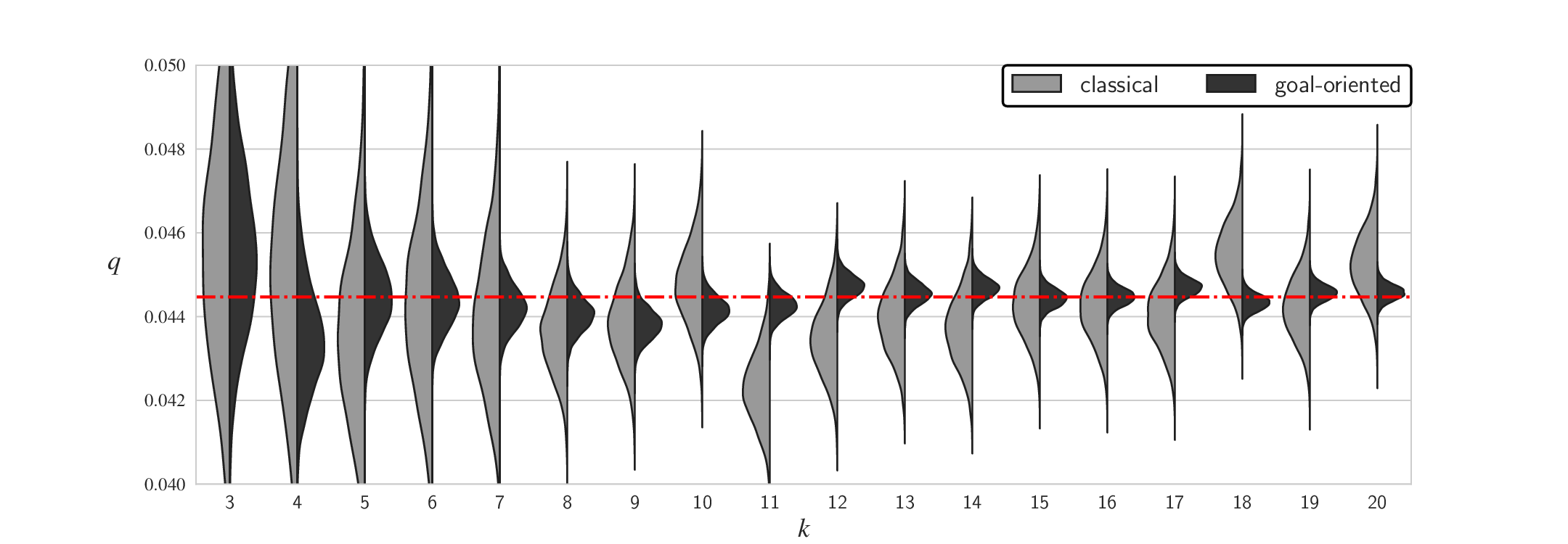}
	\caption{Goal-densities for classical and goal-oriented approaches and $k \in\{ 3, 4,\cdots, 20\}$. The dashed line is $q(\vec \mtrue)$.}
	\label{fig:violin1}
\end{figure}

Next, we compare the effectiveness of classical and goal-oriented designs in
terms of reducing the posterior variance in the goal-functional. 
Note that Theorem~\ref{theorem:variance} provides an 
analytic formula for the variance of 
the goal with respect to a given Gaussian measure. 
Here, for a vector $\vec w$ of design weights, we obtain the MAP
point $\vec\mMAPyw$ by solving the inverse problem using data 
corresponding to the active sensors.  
Then, compute the posterior variance of the goal via
\begin{equation}\label{eq:var_fun1}
	V(\vec w) := \mV_{\mu_{\po}}\crbra{\q} = 
	\ipM{\Gpo(\vec w) \mat A \vec\mMAPyw, \mat A \vec\mMAPyw} + \half\tr\big((\Gpo(\vec w)\mat A)^{2}\big).
\end{equation}

We compute $V^{1/2}$, i.e., the goal standard deviation, for designs
corresponding to the A-optimal and $\gq$-optimal designs. Additionally,
we generate $100$ random weight vectors for each $k \in \{3,\ldots,20\}$ and compute
the resulting values of $V^{1/2}$. The results of this numerical experiment are
presented in Figure~\ref{fig:goal_variance_quad}. We first observe that the
goal standard deviations corresponding to the $\gq$-optimal designs are considerably
smaller than the values for the A-optimal approach. Furthermore, both classical
and goal-oriented methods out-perform the random designs in terms of uncertainty
reduction.
Also, note the large spread in the goal standard deviations, when using random 
designs. 

\begin{figure}[ht]
	\centering
	\includegraphics[width=1\textwidth]{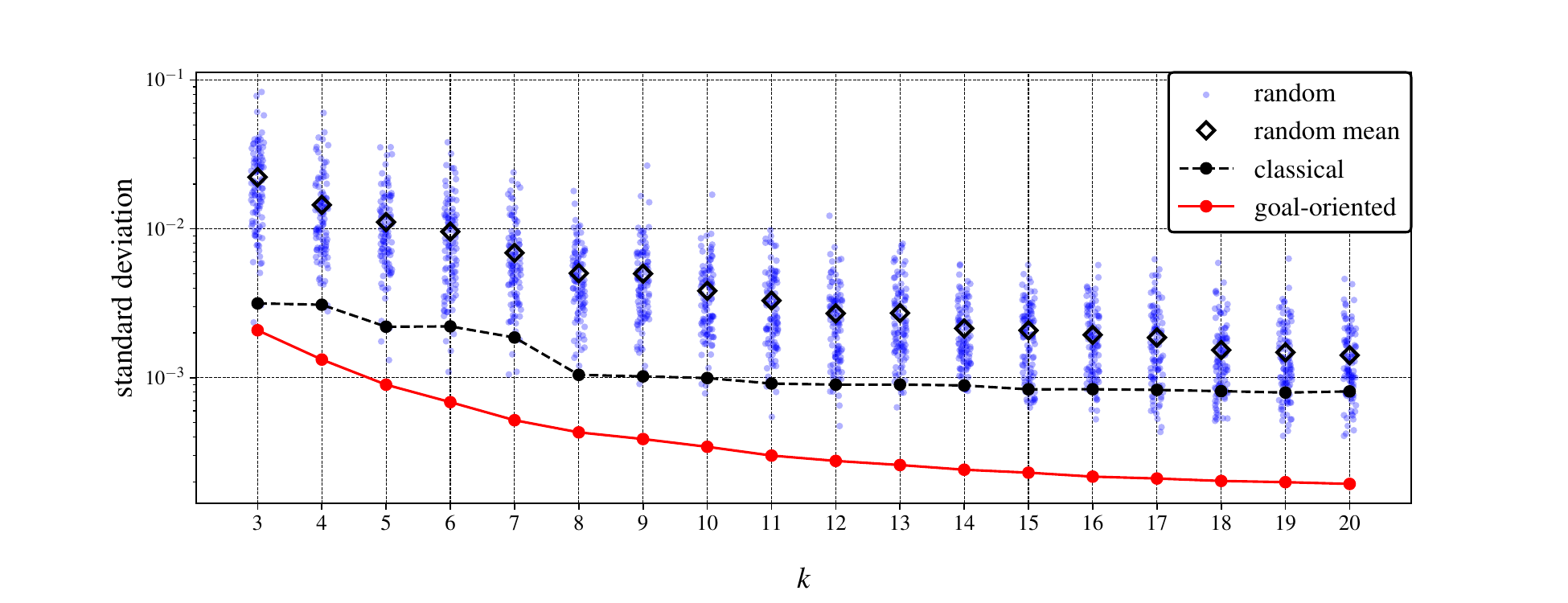}
	\caption{Standard deviations, $(V(\vec w))^{1/2}$, generated with classical (A-optimal) and goal-oriented ($\gq$-optimal) designs of size $k\in\{3,4,\dots,20\}$.}
	\label{fig:goal_variance_quad}
\end{figure}

\subsection{Model problem with a nonlinear goal functional}
\label{sec:numerics_nonlinear}
In this section, we consider an example where the goal-functional depends
nonlinearly on the inversion parameter.  

\subsubsection{Models and goal}\label{sec:nonlinear_model_and_goal}
We consider a simplified model for the flow of a tracer through a porous medium
that is saturated with a fluid.
Assuming a Darcy flow model, the 
system is governed by the 
PDEs modeling
fluid pressure $p$ and tracer concentration $c$.
The pressure equation is given by 
\begin{equation}\label{eq:model2}
\begin{alignedat}{2}
-\grad \cdot (\kap \grad p) &= m, \quad &\text{in} \ \Omega,\\
p &\equiv 0, \quad &\text{on} \ E_{0}^{p},\\
p &\equiv 1/2, \quad &\text{on} \ E_{1}^{p},\\
\grad p \cdot \vec n &\equiv 0, \quad &\text{on} \ E_{\vec n}^{p}.
\end{alignedat}
\end{equation}
Here, $\kappa$ denotes the permeability field.  The transport of the tracer is
modeled by the following steady advection diffusion equation. 
\begin{equation}\label{eq:prediction}
\begin{alignedat}{2}
-\alp \Delta c - \grad \cdot (c \kap\grad p)&= f, \quad \text{in} \ &\Omega,\\
c &\equiv 0, \quad \text{on} \ &E_{0}^{c},\\
\grad c \cdot \vec n &\equiv 0, \quad \text{on} \ &E_{\vec n}^{c}.
\end{alignedat}
\end{equation}
In this equation $\alp > 0$, is a diffusion constant and $f$ is a 
source term. Note that 
the velocity field in the transport equation 
is defined by the Darcy velocity $\vec{v} = -\kappa \nabla p$. 

In the present example, the source term $m$ in~\eqref{eq:model2} is an inversion
parameter that we seek to estimate using sensor measurements of the pressure
$p$.  Thus, the inverse problem is governed by the pressure
equation~\eqref{eq:model2}, which we call the \emph{inversion model} from now
on.  We obtain a posterior distribution for $m$ by solving this inverse problem.
This, in turn, dictates the distribution law for the pressure field $p$.
Consequently, the uncertainty in $m$ propagates into the transport equation
through the advection term in~\eqref{eq:model2}. 

We define the goal-functional by 
\begin{equation}\label{eq:goal2}
\Q(m) \defeq \int_{\Omega^{*}} c(\vec{x}; m) \, d\vec{x} = \ip{\ind(\vec x), c(\vec x; m)},
\end{equation}
where $\Omega^* \subset \Omega$ is a subdomain of interest, and $\ind$ is the indicator function of this set. 
Note that
evaluating the goal-functional requires solving the pressure
equation~\eqref{eq:model2}, followed by solving the transport 
equation~\eqref{eq:prediction}. In what follows, we call~\eqref{eq:prediction} 
the \emph{prediction model}.

Here, the domain $\Omega$ is chosen to be the unit square. In~\eqref{eq:model2},
we set $E_{0}^{p}$ as the right boundary and $E_{1}^{p}$ as the left boundary.
Additionally, $E_{\vec n}^{p}$ is selected as the union of the top and bottom
edges of $\Omega$.  
The permeability field $\kap(\vec x)$ simulates a channel or pocket of higher
permeability, oriented left-to-tight across $\Omega$. We display this field in
Figure~\ref{fig:non_lin_inverse}~(top-left). 

As for the prediction model, we take $E_{0}^{c}$ to be the union of the top,
bottom, and right edges of $\Omega$, and $E_{\vec n}^{c}$ as the left edge.  The
source $f$ in~\eqref{eq:prediction} is a single Gaussian-like function, shown in
Figure~\ref{fig:beta}, and the diffusion constant is set to $\alp = 0.12$. 
Moreover, $\Omega^*$ is given by
\[
\Omega^* = D_1 \cup D_2, \quad \text{with} \quad D_1 = [0.18, 0.32] \times [0.46, 0.68] \quad{and} \quad D_2 = [0.54, 0.75] \times [0.39, 0.75].	
\]

We require a ground-truth inversion parameter $\mtrue$ for data generation. 
This is selected as the sum of two Gaussian-like
functions, oriented asymmetrically; see Figure~\ref{fig:non_lin_inverse}~(top-right).
For the inverse problem, we set the noise variance to $\sig^{2} = 10^{-5}$,
resulting in approximately $1\%$ noise. The prior mean is set to the 
constant function $\mpr \equiv 4$.  The prior
covariance operator 
is defined according to~\eqref{eq:elliptic_PDE} with $(a_{1}, a_{2}) =
(0.8, 0.04)$. We use $N_{x} = 30^{2}$ finite element grid points and $N_{s} = 13^{2}$
equally-spaced candidate sensors. 

We depict the pressure field corresponding to the true parameter along with
the Darcy velocity, in Figure~\ref{fig:non_lin_inverse}~(bottom-left).
The MAP point obtained by solving the inverse problem using all $N_{s}$
sensors is reported in
Figure~\ref{fig:non_lin_inverse}~(bottom-right).
\begin{figure}[ht]
    \centering
        \includegraphics[width=.4\textwidth]{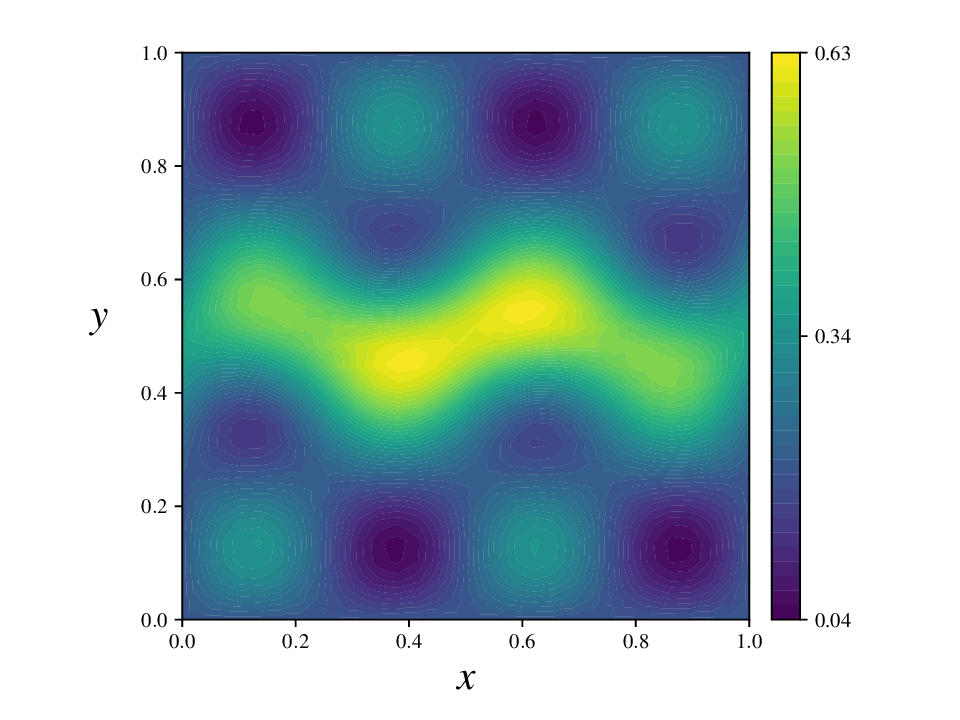}
        \includegraphics[width=.4\textwidth]{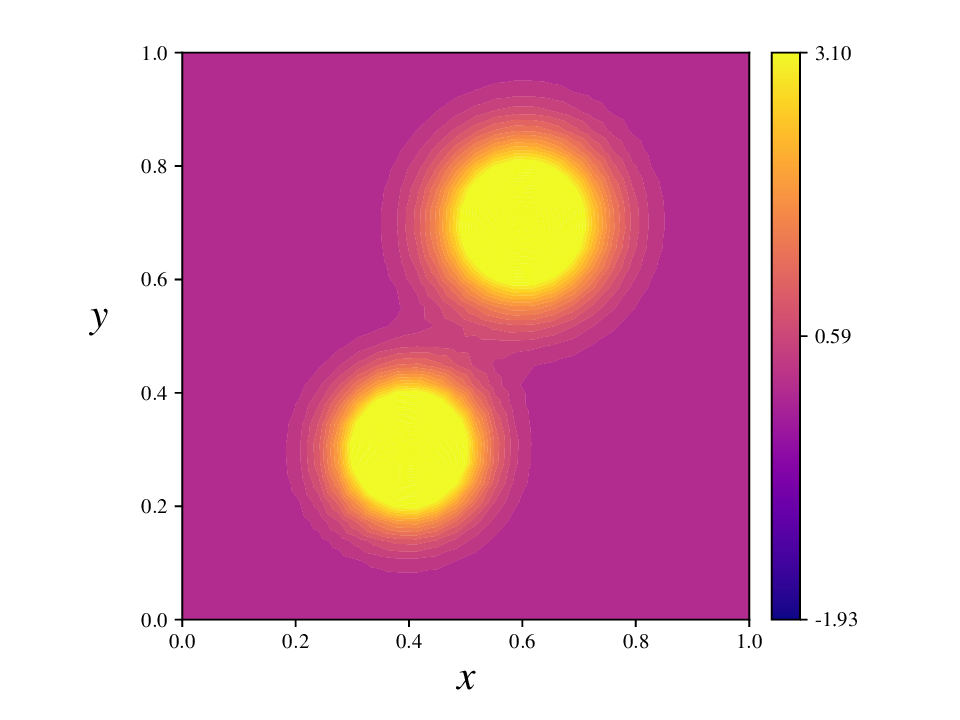}
        \includegraphics[width=.4\textwidth]{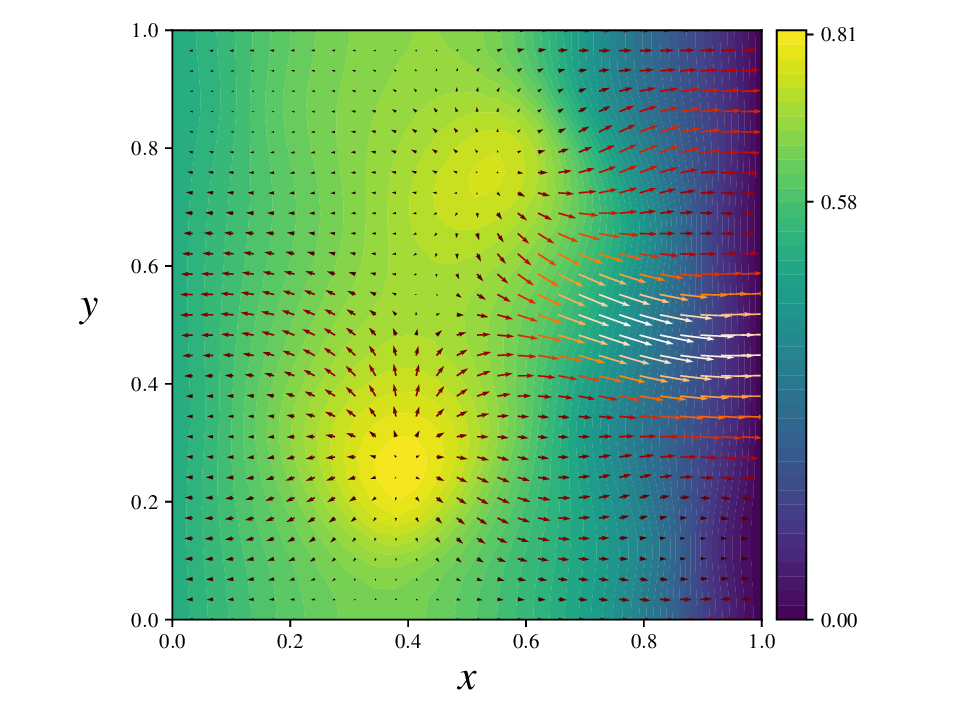}
        \includegraphics[width=.4\textwidth]{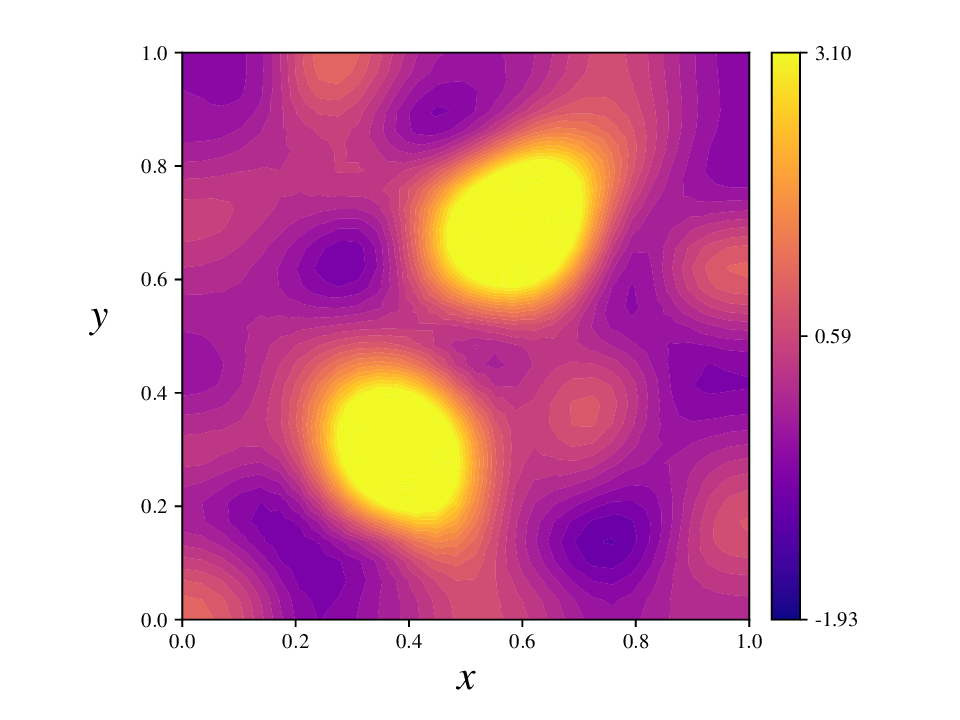}
     \caption{The true inversion parameter $\mtrue$ (top-left), permeability field $\kap(\vec x)$ (top-right), pressure solution $p(\mtrue)$ with Darcy velocity $-\kap(\vec x) \grad p$ overlayed (bottom-left), and MAP-point obtained by inverting on $N_{s}$ uniform sensors (bottom-right).}
	 \label{fig:non_lin_inverse}
\end{figure}

Recall that the goal-functional $\Q$ is formed by integrating $c$ over
$\Omega^{*}$, shown in Figure~\ref{fig:prediction}. To illustrate the dependence of $\Q$ on the inversion parameter, we plot $c(p(m))$ where $m$ is sampled from the posterior distribution. In particular, we generate a random design of size $k=3$, collect data on this design, then retrieve a posterior distribution via inversion. Figure~\ref{fig:prediction} shows $c$ corresponding to $4$ posterior samples. We overlay the subdomain $\Omega^{*}$, used to define $\Q$. 
\begin{figure}[ht]\centering
\includegraphics[width=.4\textwidth]{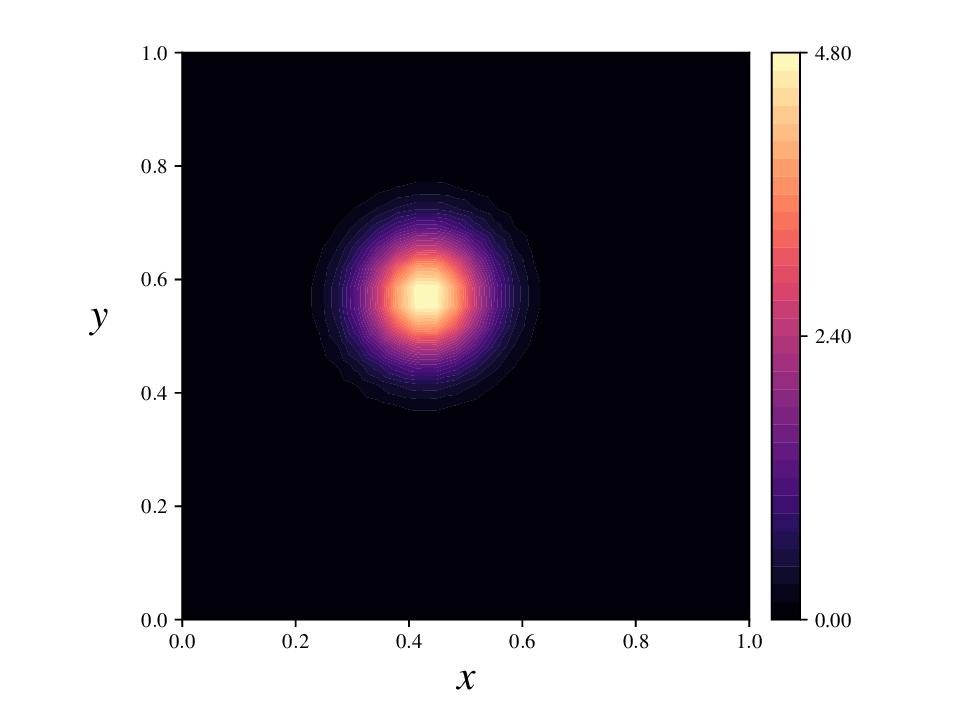}
\caption{The source term $f$ in~\eqref{eq:prediction}.}
\label{fig:beta}
\end{figure}
\begin{figure}[ht]
	\centering
	\includegraphics[width=.9\textwidth]{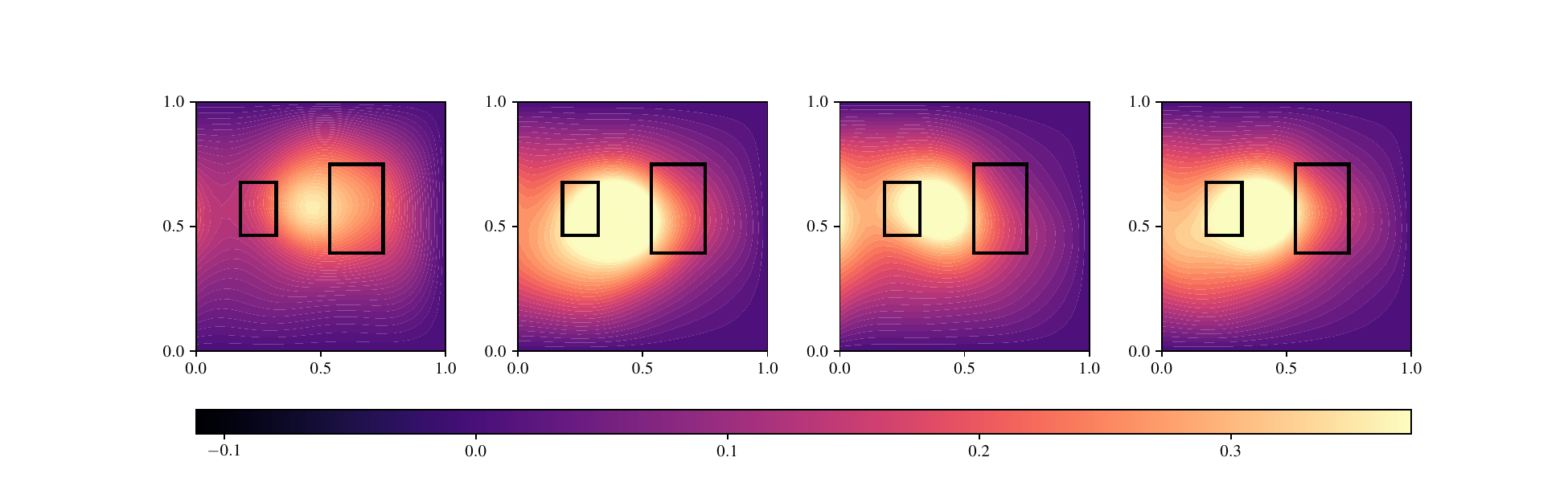}
	\caption{Concentration field $c(p(m))$ generated with $4$ posterior samples. The subdomain $\Omega^{*}$ (black rectangles) overlaid.}
	\label{fig:prediction}
\end{figure}
Note that due to the small amount of data used for solving the inverse 
problem, there is considerable variation in realizations of the 
concentration field.

\subsubsection{Optimal designs and uncertainty}
To compute $\gq$-optimal designs, we need to minimize the 
discretized goal-oriented criterion~\eqref{eq:goed_discretized}. 
The definition of this criterion
requires the first and second order derivatives of the goal-functional, 
as well as an
expansion point.  We provide the derivation of the gradient and Hessian of $\Q$
for the present example, 
in a function space setting, in Appendix~\ref{appdx:derivatives}.  As for the 
expansion point, we experiment with using the prior
mean as well as prior samples for $\mbar$.  The numerical tests that follow include testing
the effectiveness of the $\gq$-optimal designs compared to A-optimal ones, as well as
comparisons of designs obtained by minimizing the $\gl$-optimality criterion.
As before, we utilize the spectral method, outlined in
Section~\ref{sec:goed_spectral}, to estimate both classical and goal-oriented criteria. 
\begin{figure}[ht]
	\centering
	\includegraphics[width=\textwidth]{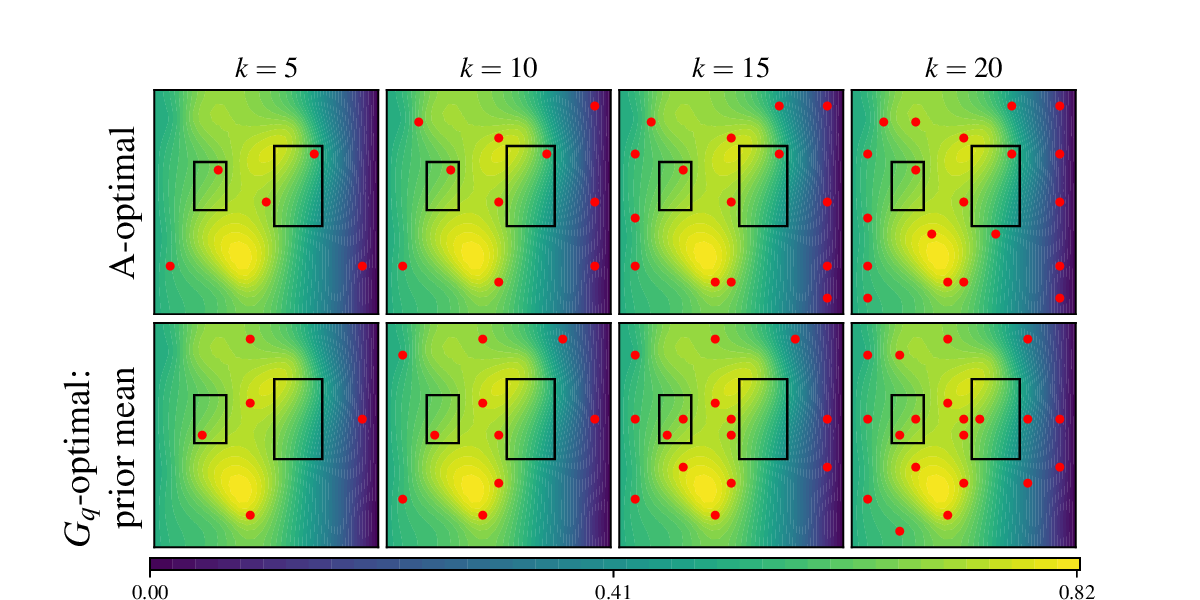}
	\caption{A-optimal and $\gq$-optimal designs of size $k \in \{5, 10, 15, 20\}$ plotted over the true pressure field $p(\mtrue)$.}
	\label{fig:classic_quadratic_designs}
\end{figure}

Figure~\ref{fig:classic_quadratic_designs} compares A-optimal and $\gq$-optimal
designs. Here, we use $\mpr$ as the expansion point to form the $\gq$-optimality criterion. 
Note that, unlike the study in Section~\ref{sec:numerics_quad}, the sensors
corresponding to the goal-oriented designs do not necessarily accumulate around
$\Omega^{*}$.  This indicates the non-trivial nature of such 
sensor placements and pitfalls of following an intuitive approach of placing
sensors within $\Omega^*$. 

Next, we examine the effectiveness of the $\gq$-optimal designs in comparison to
A-optimal ones. 
We use a prior sample as expansion point for the $\gq$-optimality criterion.
Figure~\ref{fig:violin2} presents a pairwise comparison of the goal-functional 
densities obtained by solving the Bayesian inverse problem with goal-oriented and 
classical design of various sizes. 
\begin{figure}[ht]
	\centering
	\includegraphics[width=\textwidth]{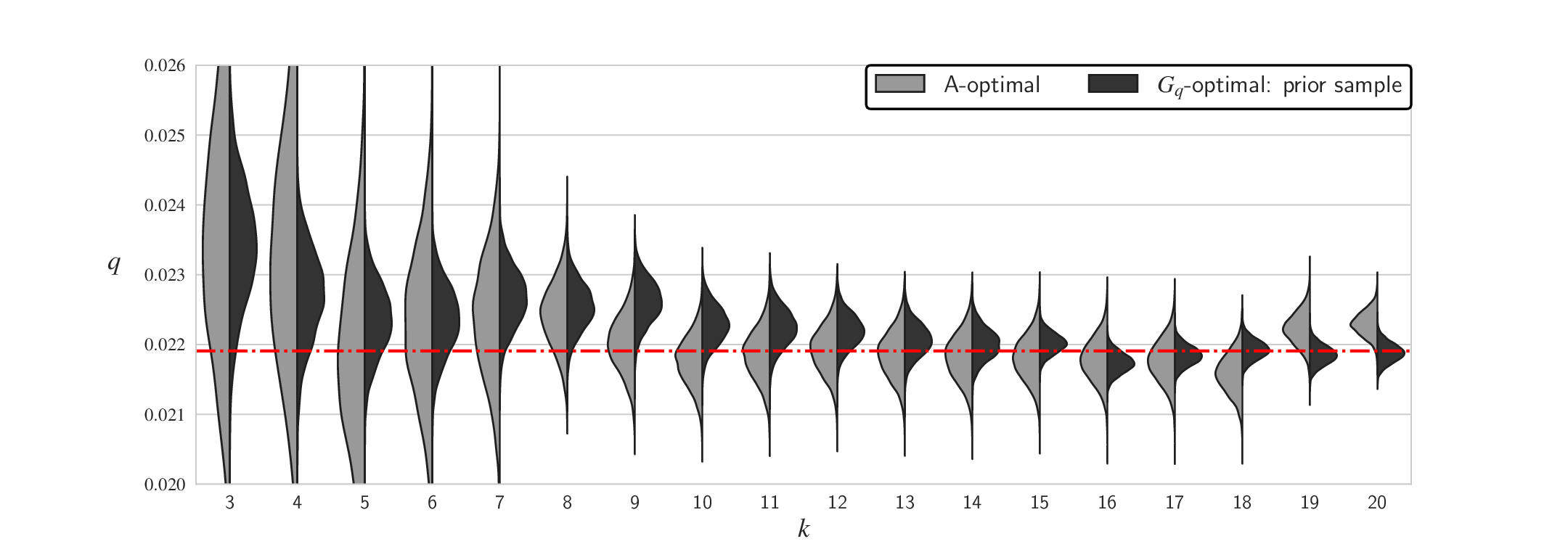}
	\caption{Goal-densities for A-optimal and $\gq$-optimal designs of size $k\in\{3,4,\dots,20\}$. The dashed line is $q(\vec \mtrue)$.}
	\label{fig:violin2}
\end{figure}
We observe that the densities corresponding to goal-oriented $\gq$-optimal designs have a
smaller spread and tend to be closer to the true goal value in comparison to
densities obtained using classical A-optimal designs.  This provides further
evidence that the proposed framework is more effective than the classical
approach in reducing the uncertainty in the goal-functional. 

\subsubsection{Comparing goal-oriented OED approaches} 
To gain insight on the $\gq$-optimality and $\gl$-optimality approaches, we report designs corresponding to these schemes in
Figure~\ref{fig:classic_goal_designs3}. Both goal-oriented criteria are built using a prior sample for expansion point.  
\begin{figure}[ht]
	\centering
	\includegraphics[width=\textwidth]{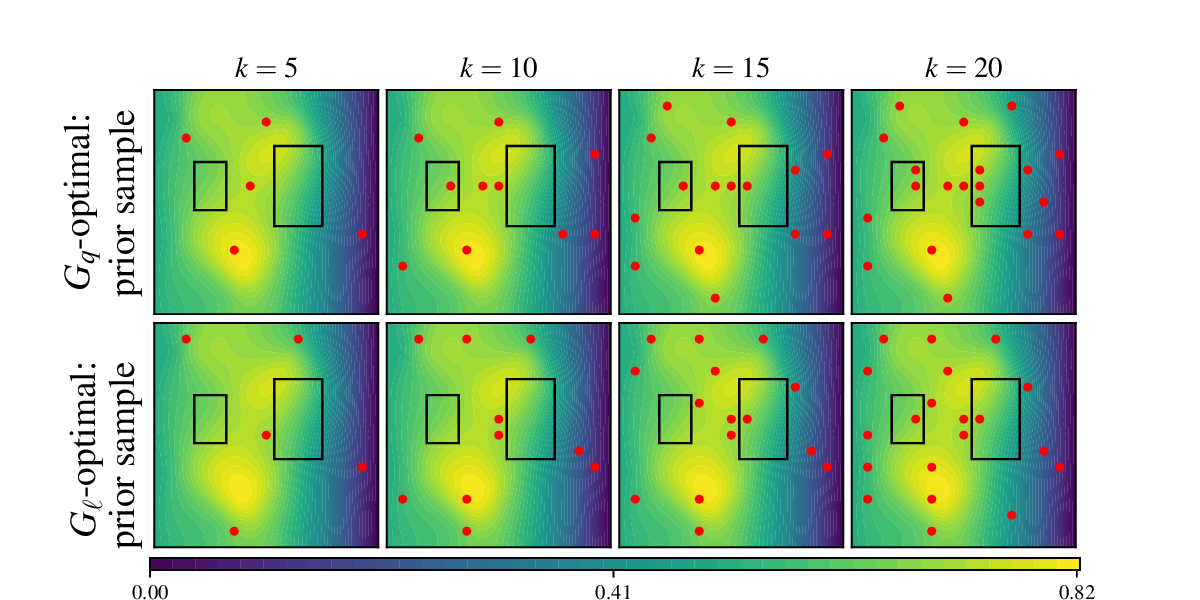}
	\caption{Designs of size $k \in \{5, 10, 15, 20\}$ corresponding to the $\gq$-optimality and $\gl$-optimality approaches plotted 
	over the true pressure field $p(\mtrue)$.}
	\label{fig:classic_goal_designs3}
\end{figure}
We note that the $\gq$-optimal and $\gl$-optimal
designs behave similarly. This is most evident for the optimal designs of size
$k=5$. To provide a quantitative comparison of these two sensor placement
strategies,  we report goal-functional densities corresponding to $\gl$-optimal and
$\gq$-optimal designs in Figure~\ref{fig:violin3}.  We use the same prior sample 
as the expansion point. Overall, we note that the
$\gq$-optimal designs are more effective in reducing the uncertainty in the
goal-functional compared to $\gl$-optimal designs. 
\begin{figure}[ht]
	\centering
	\includegraphics[width=.9\textwidth]{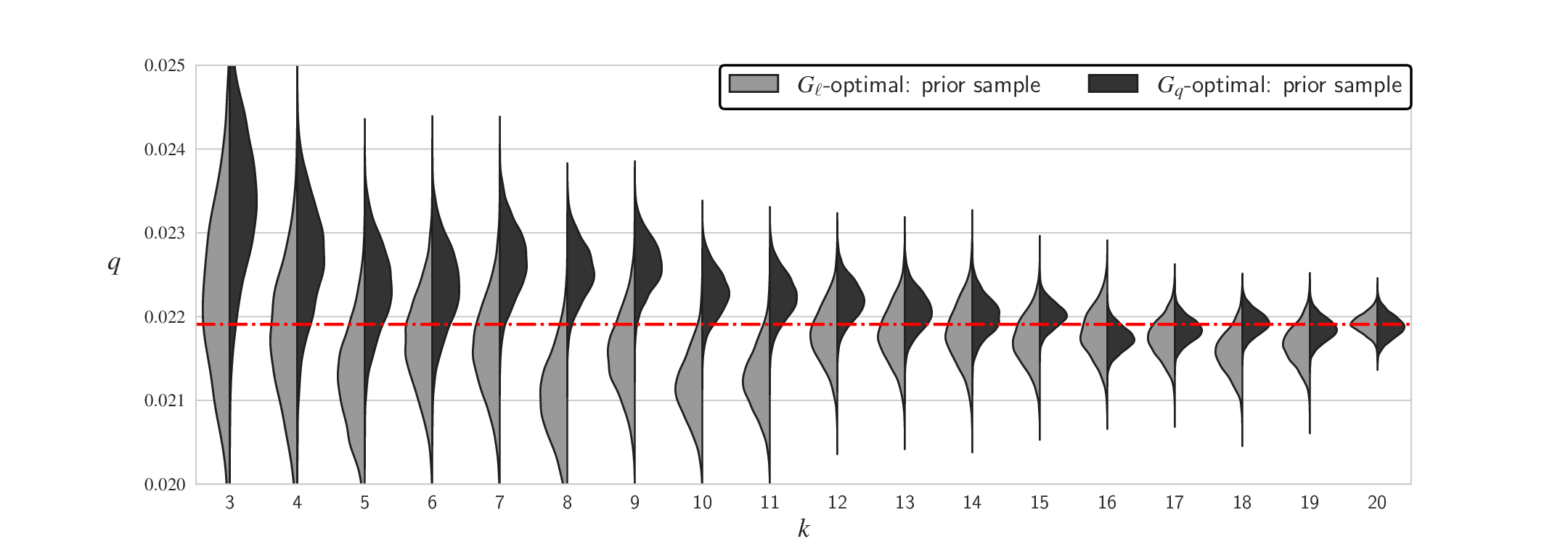}
	\caption{Goal-densities corresponding to the $\gl$-optimal and $\gq$-optimal designs of size $k \in \{1,2,\dots,20\}$, using the prior mean as expansion point. The dashed line is the true goal value.}
	\label{fig:violin3}
\end{figure}

So far, our numerical tests correspond to comparisons with a single expansion point, being either the 
prior mean or a sample from the prior. 
To understand how results vary for different expansion points, we
conduct a  
numerical experiment with multiple 
expansion points. The set of expansion points used in the following demonstration consists of the prior mean and prior samples.  
This study enables
a thorough comparison of the proposed $\gq$-optimality framework against the
classical A-optimal and goal-oriented $\gl$-optimality approaches.

\begin{figure}[ht]
	\centering
	\includegraphics[width=\textwidth]{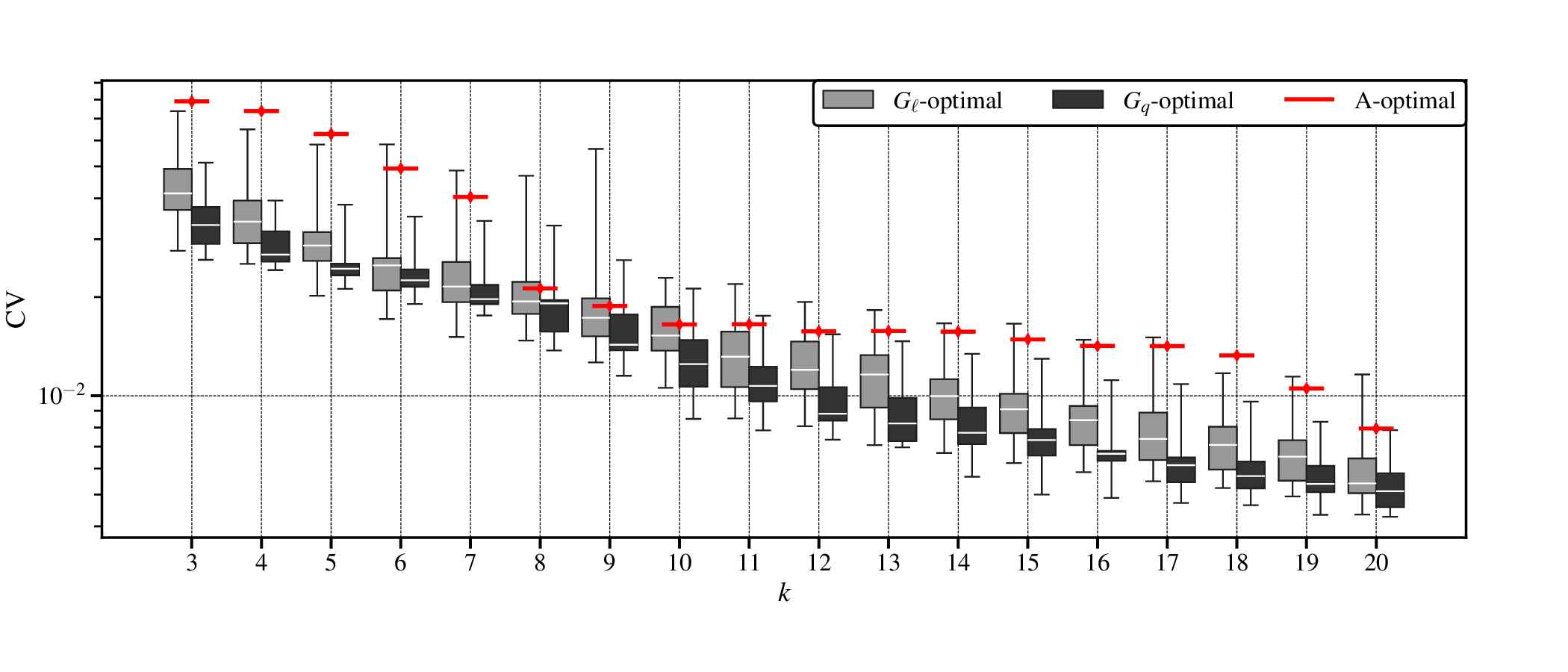}
	\caption{Coefficients of variation (CV) corresponding to A-optimal, $\gq$-optimal, and $\gl$-optimal designs of size $k \in \{3, 4, \dots, 20\}$. Goal oriented designs are obtained using the prior mean and $20$ prior samples as expansion points.}
	\label{fig:nonlinear_CV}
\end{figure}
We use the prior mean and $20$ prior samples as expansion points.  These $21$
points are used to form $21$ $\gl$-optimality and 21 $\gq$-optimality criteria.  For
each of the 21 expansion points, 
we obtain $\gl$-optimal and $\gq$-optimal optimal designs of size $k \in
\{3,4,\cdots,20\}$. This results in 42 posterior distributions corresponding 
to each of the considered values of $k$. 
To compare the performance of the two goal-oriented approaches, we 
consider a normalized notion of the posterior uncertainty 
in the goal-functional in each case. 
Specifically, we consider coefficient of variation (CV) of the goal-functional:
\[
CV(\q) \defeq \frac{\sqrt{\mathbb{V}\{\q\}}}{\mathbb{E}\{\q\}},	
\]
where $\mathbb{V}$ and $\mathbb{E}$ indicate variance and expectation 
with respect to the posterior distribution. We estimate the CV empirically.

For each $k \in \{3, \ldots, 20\}$, we obtain 21 CV values for the
goal-functional corresponding to $\gl$-optimal designs and 21 CV values for the
goal-functional corresponding to $\gq$-optimal designs.  We also compute the classical
A-optimal design for each $k$.  The results are summarized in
Figure~\ref{fig:nonlinear_CV}. For each $k$, we report the CV corresponding to
the A-optimal design of size $k$ and pairwise box plots depicting the distribution of the
CVs for the computed goal-oriented designs of size $k$. 
It is clear from Figure~\ref{fig:nonlinear_CV} that, on average, both goal-oriented designs produce smaller CVs than the classical approach. 
Furthermore, $\gq$-optimal designs reduce CV the most.
Additionally, we notice that choice of expansion point matters significantly for the goal-oriented schemes,
especially for the $\gl$-optimal designs. This is highlighted by considering the
$k=9$ case, where there is a high variance in the CVs.  To illustrate this
further, we isolate a subset of the design sizes and report the
statistical outliers in the CV data in addition to the box plots; 
see Figure~\ref{fig:nonlinear_CV_sample}.
\begin{figure}[ht]
	\centering
	\includegraphics[width=.9\textwidth]{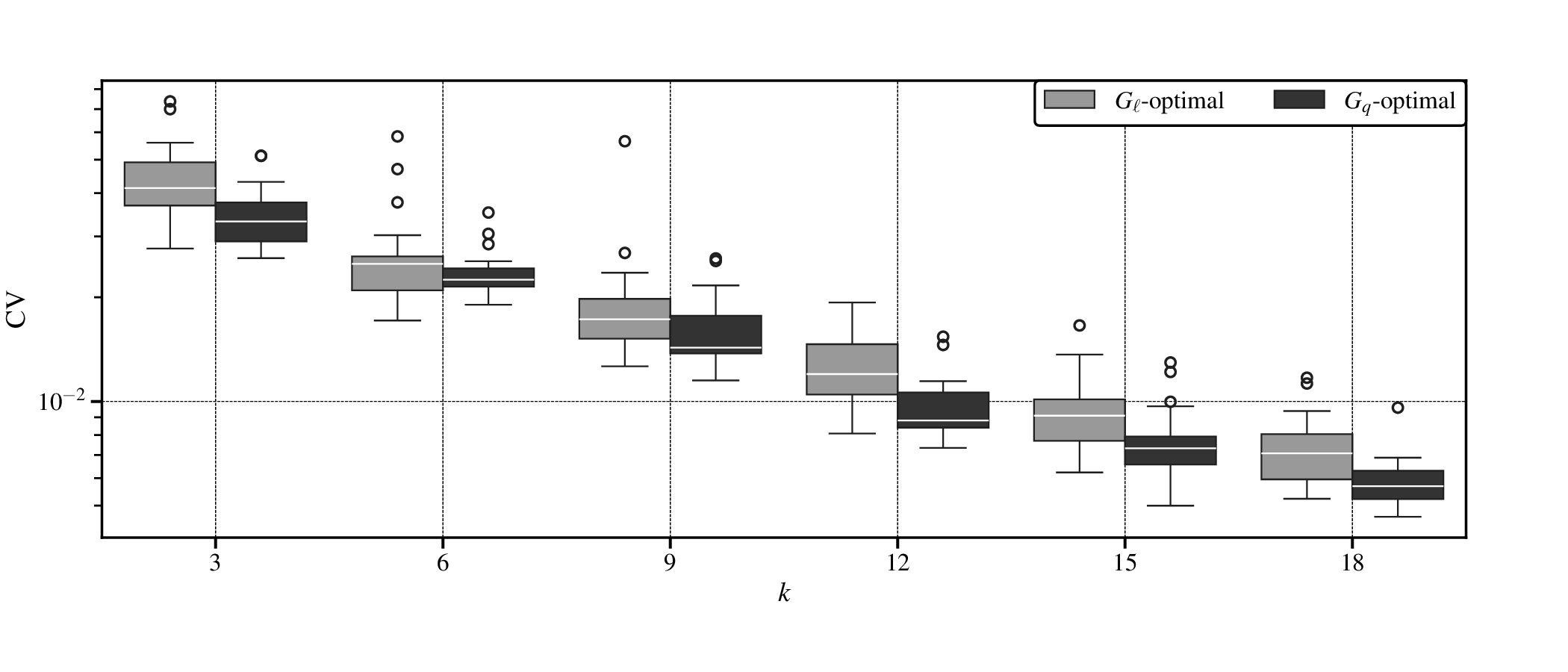}
	\caption{A subsample of the CV's corresponding to $\gl$-optimal and $\gq$-optimal designs, with outliers represented as circles.}
	\label{fig:nonlinear_CV_sample}
\end{figure}

Overall, the numerical tests paint a consistent picture: (i) both types of goal-oriented designs outperform the classical designs; (ii) compared to $\gl$-optimal
designs, the $\gq
$-optimal designs are more effective in reducing uncertainty in
the goal functional; and (iii) compared to the $\gq$-optimal designs, $\gl$-optimal
designs show greater sensitivity to the choice of the expansion point.

\section{Conclusion}
In the present work, we developed a mathematical and computational framework for
goal-oriented optimal of infinite-dimensional Bayesian linear inverse problems
governed by PDEs. The focus is on the case where the quantity of interest 
defining the goal is a nonlinear functional of the inversion parameter.  
Our framework is
based on minimizing the expected posterior variance of the quadratic
approximation to the goal-functional. We refer to this as the $\gq$-optimality criterion. 
We demonstrated that this strategy outperforms classical OED, as well as
c-optimal experimental design (which is based on linearization of the goal-functional), 
in reducing the uncertainty in the
goal-functional.  Additionally, the cost of our methods, measured in number of
PDEs solves, are independent of the dimension of the discretized
inversion parameter. 

Several avenues of interest for future investigations exist on both theoretical and
computational fronts. For one thing, it is natural to consider the case when the
inverse problem is nonlinear.
Clearly, the resulting methods would expand the application space of our goal-oriented
framework. A starting point for addressing goal-oriented optimal design 
of nonlinear inverse problems is to consider a linearization of the 
parameter-to-observable mapping, resulting in locally optimal goal-oriented 
designs. A related approach is to use a Laplace approximation to the posterior, 
as is common in optimal design of infinite-dimensional inverse problems.
Cases of inverse problems with potentially multi-modal designs might 
demand more costly strategies based on sampling. It would be interesting 
to investigate suitable importance sampling schemes in such contexts, for 
efficient evaluation of the $\gq$-optimality criterion. 

A complementary perspective on identifying measurements that are informative to
the goal-functional is a post-optimality sensitivity analysis approach. This
idea was used in~\cite{SunseriHartVanBloemenWaandersAlexanderian20} to identify
measurements that are most influential to the solution of a 
deterministic inverse
problem. Such ideas were extended to cases of Bayesian inverse problems
governed by PDEs
in~\cite{SunseriAlexanderianHartEtAl24,ChowdharyTongStadlerEtAl24}.  This
approach can also be used in a goal-oriented manner. Namely, one can consider
the sensitivity of measures of uncertainty in the goal-functional to different
measurements to identify informative experiments.  This is particularly
attractive in the case of nonlinear inverse problems governed by PDEs.

Another important line of inquiry is to investigate goal-oriented criteria defined in terms of
quantities other than the posterior variance. For example, one can seek  
designs that 
maximize information gain regarding the goal-functional or optimizing
inference of the tail behavior of the goal-functional.  A yet another potential
avenue of further investigations is considering relaxation strategies to replace
the binary goal-oriented optimization problem with a continuous optimization
problem, for which powerful gradient-based optimization methods maybe deployed.

\section*{Acknowledgments}
This article has been authored by employees of National Technology \&
Engineering Solutions of Sandia, LLC under Contract No.~DE-NA0003525 with the
U.S.~Department of Energy (DOE). The employees own all right, title and interest
in and to the article and are solely responsible for its contents.
SAND2024-15167O. 

This material is also based upon work supported by the U.S. Department of Energy,
Office of Science, Office of Advanced Scientific Computing Research Field Work
Proposal Number 23-02526.  

\bibliographystyle{abbrv}
\bibliography{refs}

\appendix

\section{Proof of Theorem~\ref{theorem:variance}}
\label{appdx:variance_of_quad}
We first recall some notations and definitions regarding 
Gaussian measures on Hilbert spaces.  
Recall that for a Gaussian measure $\mu = \mf{N}(m_0, \iC)$
on a real separable Hilbert space $\ms{M}$ the mean $m_0 \in \ms{M}$ satisfies, 
\[
    \ip{m_0, v} = \int_\ms{M} \ip{s, v} \mu(ds), 
	\quad \text{for all } v \in \ms{M}.
\]
Moreover, $\iC$ is a positive self-adjoint trace 
class operator that satisfies, 
\begin{equation}\label{equ:covariance_definition}
\ip{\iC u, v} = \int_\ms{M} \ip{u, s - m_0}\ip{v, s - m_0} \, \mu(ds).
\end{equation}
For further details, see~\cite[Section 1.4]{DaPrato06}.
We assume that $\iC$ is strictly positive. 
In what follows, we let $\{ e_i\}_{i \in \mathbb{N}}$ 
be the complete orthonormal set of eigenvectors of 
$\iC$ and $\{\lambda_i\}_{i \in \mathbb{N}}$ 
the corresponding (real and positive) eigenvalues.

Consider the probability space $(\ms{M}, \ms{B}(\ms{M}), \mu)$, 
where $\ms{B}$ is the Borel $\upsigma$-algebra on $\ms{M}$.
For a fixed $v \in \ms{M}$, the linear functional 
\[
\varphi(s) = \ip{s, v}, \quad s \in \ms{M},
\]
considered as a random variable $\varphi:(\ms{M}, \ms{B}(\ms{M}), \mu)
\to (\mathbb{R}, \ms{B}(\mathbb{R}))$, is a Gaussian random variable 
with mean $\varphi(m_0)$ and variance $\sigma^2_v = \ip{\iC v, v}$.
More generally, for $v_1, \ldots, v_n$ in $\ms{M}$, the random $n$-vector 
$\vec{Y}:\ms{M} \to \R^n$ given by 
$\vec{Y}(s) = [\ip{s, v_1}\; \ip{s, v_2}\; \ldots\; \ip{s, v_n}]^\top$ 
is an $n$-variate Gaussian whose distribution law is
$\mf{N}(\bar{\vec{y}}, \mat{C})$, 
\begin{equation}\label{equ:Y_rv}
\bar{y}_i = \ip{m_0, v_i}, \quad 
C_{ij} = \ip{\iC v_i, v_j},
\quad i, j \in \{1, \ldots, n\}.
\end{equation}

The arguments in this appendix rely heavily on the standard 
approach of using finite-dimensional projections to facilitate 
computation of Gaussian integrals. As such we also need 
some basic background results regarding Gaussian random vectors.
In particular, we need the following result~\cite{Withers85,Holmquist88,Triantafyllopoulos02}. 
\begin{lemma}\label{lemma:moments_basic}
Suppose $\vec{Y} \sim \mf{N}(\vec{0}, \mat{C})$ is an $n$-variate Gaussian random variable. Then,
for $i, j, k$, and $\ell$ in $\{1, \ldots, n\}$,
\begin{enumerate}[(a)]
\item $\mE\{Y_i Y_j Y_k\} = 0$; and  
\item $\mE\{Y_i Y_j Y_k Y_\ell\} = C_{ij} C_{k\ell} + C_{ik} C_{j\ell} + C_{i \ell}C_{jk}$.
\end{enumerate}
\end{lemma}

The following technical result is useful in what follows.

\begin{lemma}
\label{lemma:moments1}
Let $\iA$ be a bounded selfadjoint linear operator on a Hilbert space $\ms{M}$, 
and let  $\mu_{0} := \mf{N}(0, \iC)$ be a Gaussian measure on $\ms{M}$.
We have
\begin{enumerate}[(a),itemsep=0.3cm]
\item $\int_{\ms{M}}\ip{b,s}\ip{c,s} \, \mu_0(ds)= \ip{\iC b, c}$, for all $b$ and $c$ in 
$\ms{M}$;
\item $\int_{\ms{M}} \ip{\iA s, s}\ip{b, s} \, \mu_0(ds) = 0$, for all $b \in \ms{M}$; and 
\item $\int_{\ms{M}} \ip{\iA s, s}^{2} \, \mu_0(ds) = \big(\tr(\iC\iA)\big)^{2} + 2\tr\big((\iC\iA)^{2}\big)$.
\end{enumerate}
\end{lemma}
\begin{proof}
The first statement follows immediately from~\eqref{equ:covariance_definition}.
We consider (b) next. 
For $n \in \mb{N}$, we define the projector $\pi_{n}$ in terms of the eigenvectors of $\iC$, 
$$\pi_{n}(s) := \sum_{i=1}^{n}\ip{s, e_{i}}e_{i}, \quad s \in \ms{M}.$$
Note that $\vec{Y}:(\ms{M}, \ms{B}(\ms{M}), \mu_0) \to (\R^n, \ms{B}(\R^n))$ 
defined by  
$\vec{Y}(s) \defeq 
[\ip{s, e_1}\; \ip{s, e_2}\; \ldots\; \ip{s, e_n}]^\top$ has an $n$-variate 
Gaussian law, $\vec{Y} \sim \mf{N}(\vec{0}, \mat{C})$, with 
\[
C_{ij} = \ip{\iC e_i, e_j} = \lambda_{i}\delta_{ij},
\]
where $\delta_{ij}$ is the Kronecker delta. 
We next consider,
$$\ip{\iA s, s}\ip{b, s} = \limn{n}\tau_{n}(s) \quad \text{where} \quad \tau_{n}(s) \defeq \ip{\iA \pi_{n}(s), \pi_{n}(s)}\ip{b, \pi_{n}(s)}.$$
Note that for each $n \in \mathbb{N}$, 
\[
\int_{\ms{M}} \tau_n(s)\, \mu_0(ds) 
= \sum_{i,j,k=1}^n \ip{\iA e_i, e_j} \ip{b, e_k} \int_{\ms{M}} \ip{s,e_i}\ip{s,e_j}\ip{s,e_k} \, \mu_0(ds) 
= \sum_{i,j,k=1}^n \ip{\iA e_i, e_j} \mathbb{E}\{ Y_i Y_j Y_k\} = 0.
\]
The last step follows from Lemma~\ref{lemma:moments_basic}(a).
Furthermore, 
$|\tau_{n}(s)| \leq \|\iA\|\|b\|\|\pi_{n}(s)\|^{3} \leq \|\iA\|\|b\|\|s\|^{3}$, for all $n \in \mN$.
Therefore, since $\int_{\ms{M}} \| s \|^3 \, \mu_0(ds) < \infty$, by 
the Dominated Convergence Theorem,
\[
\int_{\ms{M}} \ip{\iA s,s}\ip{b, s}   \, \mu_0(ds) 
= \int_{\ms{M}} \limn{n} \tau_{n}(s) \,\mu_0(ds) 
= \limn{n} \int_{\ms{M}} \tau_n(s)   \, \mu_0(ds) 
= 0.
\]
We next consider the third statement of the lemma. The approach is similar to the 
proof of part (b). We note that  
$$\ip{\iA s, s}^{2} = \limn{n} \theta_{n}(s) \quad \text{where} \quad \theta_{n}(s):= \ip{\iA \pi_{n}(s), \pi_{n}(s)}^{2}.$$
As in the case of $\tau_n$ above, we can easily bound $\theta_n$. Specifically, 
$|\theta_{n}(s)| \leq \|\iA\|^{2}\|s\|^{4}$, for all $n \in \mathbb{N}$.
Note also that $\int_{\ms{M}} \| s\|^4 \, \mu_0(ds) < \infty$.
Therefore, by the Dominated Convergence Theorem,
\begin{equation}\label{equ:limit_DCT_theta}
\int_{\ms{M}}\ip{\iA s, s}^{2} \, \mu_0(ds) = \int_{\ms{M}} \limn{n} \theta_{n}(s) \,\mu_0(ds) = \limn{n} \int_{\ms{M}} \theta_{n}(s) \,\mu_0(ds).
\end{equation}
Next, note that for each $n$, 
\begin{align}
\int_{\ms{M}} \theta_{n}(s) \,\mu_0(ds)  
&= 
\summ{i}{j}{k}{\ell} 
	\Ae{i}{j} \Ae{k}{\ell}\int_{\ms{M}}\ip{s,e_i}\ip{s,e_j}\ip{s,e_k}\ip{s,e_\ell} \, \mu_0(ds) \notag\\
&= 
	\summ{i}{j}{k}{\ell} \Ae{i}{j} \Ae{k}{\ell}(C_{ij} C_{k\ell} + C_{ik} C_{j\ell} + C_{i \ell}C_{jk}), 
	\label{equ:ugly}
\end{align}
where we have used Lemma~\ref{lemma:moments_basic}(b) in the final step.
Let us consider each of the three terms in~\eqref{equ:ugly}. We note,
\begin{align}
\sum_{i,j,k,\ell=1}^{n}\ip{\iA e_{i}, e_{j}}\ip{\iA e_{k}, e_{\ell}} C_{ij}C_{k\ell} 
&= \sum_{i,j,k,\ell}^{n}\lam_{j}\lam_{k}\ip{\iA e_{i}, e_{j}}\ip{\iA e_{k}, e_{\ell}}\delta_{ij}\delta_{k\ell} \notag\\ 
\notag &= \sum_{i,k=1}^{n}\lam_{i}\lam_{k}\ip{\iA e_{i}, e_{i}}\ip{\iA e_{k}, e_{k}}\\
&= \paren{\sum_{i=1}^{n}\lam_{i}\ip{\iA e_{i}, e_{i}}}^{2}\notag\\ 
&= \paren{\sum_{i=1}^{n}\ip{\iA e_{i,}, \iC e_{i}}}^{2} \to \big(\tr(\iC \iA)\big)^2, 
\quad\text{as } n \to \infty. \label{equ:ugly_a}
\end{align}
Before, we consider the second and third terms in~\eqref{equ:ugly}, we note 
\begin{multline*}
	\tr\big((\iA\iC)^2\big) = \tr(\iA\iC\iA\iC) = \sum_{i=1}^\infty \ip{e_i, \iA \iC \iA \iC e_i} = 
	\sum_{i=1}^\infty \lambda_i \ip{e_i, \iA \iC \iA e_i}
	= \sum_{i=1}^\infty \lambda_i \ip{e_i, \iA \iC \big(\sum_{j=1}^\infty \ip{\iA e_i, e_j}e_j\big)}\\
	= \sum_{i,j=1}^\infty \lambda_i \ip{e_i, \iA \iC  e_j} \ip{\iA e_i, e_j}
	= \sum_{i,j=1}^\infty \lambda_i\lambda_j \ip{\iA e_i, e_j}^2.
 \end{multline*}
Next, we note 
\begin{equation}\label{equ:ugly_b}
\sum_{i,j,k,\ell=1}^{n}\ip{\iA e_{i}, e_{j}}\ip{\iA e_{k}, e_{\ell}}C_{ik}C_{j\ell} 
= \sum_{i,j,k,\ell=1}^{n} \lam_{i}\lam_{j}\ip{\iA e_{i}, e_{j}}\ip{\iA e_{k}, e_{\ell}}\delta_{ik}\delta_{j\ell}
= \sum_{i,j=1}^{n} \lam_{i}\lam_{j}\ip{\iA e_{i}, e_{j}}^{2} \to \tr\big((\iA\iC)^2\big), 
\end{equation}
as $n \to \infty$. A similar argument shows, 
\begin{equation}\label{equ:ugly_c}
\sum_{i,j,k,\ell=1}^{n}\ip{\iA e_i, e_j}\ip{\iA e_k, e_\ell} C_{i\ell}C_{jk} 
= \sum_{i,j=1}^{n} \lam_{i}\lam_{j}\ip{\iA e_{i}, e_{j}}^{2} \to \tr\big((\iA\iC)^2\big),
\quad\text{as } n \to \infty.
\end{equation}
Hence, combining \eqref{equ:limit_DCT_theta}--\eqref{equ:ugly_c}, we obtain
\[
\int_{\ms{M}}\ip{\iA s, s}^{2} \, \mu_0(ds) = \int_{\ms{M}} \limn{n} \theta_{n}(s) \,\mu_0(ds) = \limn{n} \int_{\ms{M}} \theta_{n}(s) \,\mu_0(ds)
= \big( \tr(\iC\iA) \big)^2 + 2 \tr\big((\iC\iA)^2\big), 
\]
which completes the proof of the lemma.
\qed
\end{proof}

The following result extends Lemma~\ref{lemma:moments1} to the 
case of non-centered Gaussian measures. 
\begin{lemma}
\label{lemma:moments2}
Let $\iA$ be a bounded selfadjoint linear operator on a real Hilbert space $\ms{M}$, 
and let $\mu := \mf{N}(m_0, \iC)$ be a Gaussian measure on $\ms{M}$.
\begin{enumerate}[(a),itemsep=0.3cm]
\item 
$\int_\ms{M} \ip{b,s}\ip{c,s} \, \mu(ds) = \ip{\iC b, c} +\ip{b, m_{0}}\ip{c,m_{0}}$, 
for all $b$ and $c$ in $\ms{M}$;
\item 
$\int_\ms{M} \ip{\iA s, s}\ip{b, s} \, \mu(ds) = 
\big(\!\ip{\iA m_{0}, m_{0}} + \tr(\iC\iA) \big) \ip{b, m_0} + 2\ip{\iC\iA m_{0}, b}$ 
for all $b \in \ms{M}$; and
\item
$\int_\ms{M} \ip{\iA s, s}^{2} \, \mu(ds) = \big(\tr(\iC\iA)\big)^{2} + 2\tr\big((\iC\iA)^{2}\big) + 4 \ip{\iC\iA m_{0}, \iA m_{0}}
+ \paren{\ip{\iA m_{0}, m_{0}} + 2\tr\paren{\iC\iA}}\ip{\iA m_{0},m_{0}}$.
\end{enumerate}
\end{lemma}
\begin{proof}
These identities follow from Lemma~\ref{lemma:moments1} and some 
basic manipulations. For brevity, we only prove the third statement. The 
other two can be derived similarly. 
Using Lemma~\ref{lemma:moments1}(c),
\begin{align*}
\int_{\ms{M}} 
\ip{\iA (s-m_{0}), s-m_{0}}^{2} \, \mu(ds) 
	&= \big(\tr(\iC\iA)\big)^{2} + 2\tr\big((\iC\iA)^{2}\big)\\
	&= \int_{\ms{M}} \ip{\iA s, s}^{2} \, \mu(ds) 
		+ 4\int_{\ms{M}} \ip{\iA m_{0}, s}^{2} \, \mu(ds)
		+ \ip{\iA m_{0}, m_{0}}^{2}\\
		&\quad-4 \int_{\ms{M}} \ip{\iA s, s}\ip{\iA m_{0}, s} \, \mu(ds)
		-4 \int_{\ms{M}} \ip{\iA m_0, s}\ip{\iA m_{0}, m_{0}} \, \mu(ds)\\
		&\quad+ 2 \int_{\ms{M}} \ip{\iA s, s}\ip{\iA m_{0}, m_{0}} \, \mu(ds).
\end{align*}
Subsequently, we solve for $\int_{\ms{M}} \ip{\iA s, s}^{2} \, \mu(ds)$.
To do this, we require the 
formula for the expected value of a quadratic form on 
a Hilbert space (see~\cite[Lemma 1]{AlexanderianGhattasEtAl16}) and items (a) and (b) 
of the present lemma. Performing the calculation, we arrive at
\[
\int_{\ms{M}} \ip{\iA m, m}^{2} \, \mu(ds) = \big(\tr(\iC\iA)\big)^{2} + 2\tr\big((\iC\iA)^{2}\big) + 4 \ip{\iC\iA m_{0}, \iA m_{0}}
+ \paren{\ip{\iA m_{0}, m_{0}} + 2\tr\paren{\iC\iA}}\ip{\iA m_{0},m_{0}}.
\]
\qed
\end{proof}

We now have all the tools required to prove Theorem~\ref{theorem:variance}.
\begin{proof}[Proof of Theorem~\ref{theorem:variance}]
Consider~\eqref{equ:quadratic_functional_def}. 
Note that
\begin{equation}\label{eq:var_identity}
	\mV_{\mu}\crbra{\Q} = \mE_{\mu}\crbra{\Q^2} - \big(\mE_{\mu}\{\Q\}\big)^2.
\end{equation}
The second term of \eqref{eq:var_identity} 
is straightforward to compute. 
Specifically,
\begin{equation}\label{eq:2nd_exp}
\mE_{\mu}\{\Q\} =
\int_{\ms{M}} \Q(m) \, \mu(dm) 
= \frac12\int_{\ms{M}} \ip{\iA m, m} \mu(dm) + \int_{\ms{M}} \ip{b,m} \, \mu(dm)
= \frac{1}{2}\paren{\ip{\iA m_{0}, m_{0}} + \tr\paren{\iC\iA}} + \ip{b, m_{0}}.
\end{equation}
Computing the first term in~\eqref{eq:var_identity} is 
facilitated by  Lemma~\ref{lemma:moments2}. We note
\[
\begin{alignedat}{2}
\mE_{\mu}\{\Q^2\} &= \int_{\ms{M}} \Q(m)^2 \, \mu(dm) \\
&= 
\frac{1}{4}\int_{\ms{M}} \ip{\iA m, m}^{2} \, \mu(dm) 
+ \int_{\ms{M}} \ip{\iA m, m}\ip{b, m} \, \mu(dm)
+ \int_{\ms{M}} \ip{b,m}^{2} \, \mu(dm) \\
&= \frac{1}{4}\ip{\iA m_{0}, m_{0}}^{2} + \ip{\iA m_{0}, m_{0}}\ip{b,m_{0}} + \ip{\iC\iA m_{0}, \iA m_{0}}\\
&\quad+ 2\ip{\iC\iA m_{0}, b} + \ip{b, m_{0}}^{2} + \ip{\iC b, b}\\
&\quad+ \paren{\ip{\iA m_{0}, m_{0}} + \ip{b,m_{0}}}\tr\paren{\iC\iA}\\
&\quad+\frac{1}{4}\big(\tr(\iC\iA)\big)^{2} + \half\tr\big((\iC\iA)^{2}\big).
\end{alignedat}
\]
Substituting this and \eqref{eq:2nd_exp} into \eqref{eq:var_identity} and simplifying 
provides the desired identity for $\mathbb{V}_{\mu}\{\Q\}$.
\qed
\end{proof}

\section{Proof of Theorem~\ref{theorem:main}}
\label{sec:proof_main_thm}

\begin{proof}[Proof of Theorem~\ref{theorem:main}]
We begin with the following definitions
\begin{equation}\label{eq:Abc}
\iA \defeq \ddQ, \quad b \defeq \dQ - \ddQ\mbar, \quad 
c \defeq \Q(\mbar) - \ip{\dQ, \mbar} + \frac12\ip{\ddQ \mbar, \mbar}.
\end{equation}
These components enable expressing $\Q_{\qua}$ as 
\[
	\Q_{\qua}(m) = \frac{1}{2}\ip{\iA m, m} + \ip{b, m} + c.
\]
Note that the variance of $\Q_{\qua}$ does not depend on $c$. 
We can apply Theorem~\ref{theorem:variance} to obtain an expression for the variance of $\Q_{\qua}$ in relation to $\mu_{\po}$:
\begin{equation}\label{eq:moment1}
	\mV_{\mu_{\po}}\crbra{\Q_{\qua}} = \ip{\Cpo(\iA\mMAPy +b), \iA\mMAPy+b} + \half\tr\big((\Cpo\iA)^{2}\big).
\end{equation}
Next, we compute the remaining expectations in~\eqref{eq:psi}. However, this will require some manipulation of the formula for $\mMAPy$. 
We view the MAP point, given in~\eqref{eq:post}, as an affine transformation on data $\vec y$. Thus,
\begin{equation}\label{eq:affine}
	\mMAPy = \mapA \vec y + \mapB, \quad \text{where} \quad \mapA := \sig^{-2}\Cpo\iF^{*} \quad \text{and} \quad \mapB := \Cpo\Cpr^{-1}\mpr.
\end{equation}
Using this representation of $\mMAPy$, \eqref{eq:moment1} becomes
\begin{align*}
	\mV_{\mu_{\po}}\crbra{\Q_{\qua}} &= \ip{\Cpo\iA\mapA\vec 	y, \iA\mapA\vec y} + 2\ip{\Cpo\iA\mapA\vec y, \iA \mapB +b}
	\\
	&+ \ip{\Cpo(\iA\mapB + b), \iA\mapB+b} + \half\tr\big((\Cpo\iA)^{2}\big).
\end{align*}

Now the variance expression is in a form suitable for calculating the final
moments. Recalling the definition of the likelihood measure, we find that
\begin{align*}
	\mE_{{\vec y | m}} \big\{ \mV_{\mu_{\po}}\{\Q_{\qua}\}\big\} &= 		\ip{\Cpo\iA\mapA\iF m, \iA\mapA\iF m} + 2\ip{\Cpo\iA\mapA\iF m, \iA \mapB +b}\\
	&+ \ip{\Cpo(\iA\mapB + b), \iA\mapB+b} + \sig^{2}\tr\sqbra{\mapA^{*}\iA\Cpo\iA\mapA}\\
	&+ \half\tr\big((\Cpo\iA)^{2}\big).
\end{align*}
Computing the outer expectation with respect to the 
prior measure yields,
\begin{equation}\label{eq:moments}
\begin{alignedat}{1}
	\Psi &= \ip{\Cpo\iA\mapA\iF \mpr, \iA\mapA\iF \mpr} + 2\ip{\Cpo\iA\mapA\iF \mpr, \iA \mapB +b}\\
	&+ \ip{\Cpo(\iA\mapB + b), \iA\mapB+b} + \tr\sqbra{\Cpr\iF^{*}\mapA^{*}\iA\Cpo\iA\mapA\iF}\\
	&+ \sig^{2}\tr\sqbra{\mapA^{*}\iA\Cpo\iA\mapA} + \half\tr\big((\Cpo\iA)^{2}\big).
\end{alignedat}
\end{equation}

The remainder of the proof involves the acquisition of a meaningful representation of $\Psi$. Our first step towards simplification requires substituting the components $\mapA$ and $\mapB$ of $\mMAPy$, given by 
\eqref{eq:affine}, into \eqref{eq:moments}. We follow this by recognizing occurrences of $\iHmis$ in the resulting expression. Performing these operations, we have that
\begin{align}
	\Psi &= \ip{\Cpo\iA\Cpo\iHmis\mpr, \iA\Cpo\iHmis\mpr} 			  \tag{\ensuremath{A_{1}}}\\
	&+ 2\ip{\Cpo\iA\Cpo\iHmis\mpr, \iA\Cpo\Cpr^{-1}\mpr+b}			  \tag{\ensuremath{A_{2}}}\\
	&+ \ip{\Cpo\paren{\iA\Cpo\Cpr^{-1}\mpr+b},\iA\Cpo\Cpr^{-1}\mpr+b} \tag{\ensuremath{A_{3}}}\\
	&+ \tr\paren{\iHmis\Cpr\iHmis\Cpo\iA\Cpo\iA\Cpo}				  \tag{\ensuremath{B_{1}}}\\
	&+ \tr\paren{\iHmis\Cpo\iA\Cpo\iA\Cpo}							  \tag{\ensuremath{B_{2}}}\\
	&+ \frac12 \tr\big((\Cpo\iA)^{2}\big)								  \tag{\ensuremath{B_{3}}}.
\end{align}
To facilitate the derivations that follow, we have assigned a label to each of
the summands, in the above equation.  We refer to
$A_{1}, A_{2}$, and $A_{3}$ as \emph{product terms} and call $B_{1}$, $B_{2}$, and 
$B_{3}$ the \emph{trace terms}. 

Let us consider the product terms. Note that
\begin{align}
	A_{1} + A_{2} + A_{3} &= \ip{\Cpo\iA\Cpo\iHmis\mpr,\iA\Cpo\iHmis\mpr} \tag{\ensuremath{A_{1}}}\\
	&+ 2\ip{\Cpo\iA\Cpo\iHmis\mpr, \iA\Cpo\Cpr^{-1}\mpr}			      \tag{\ensuremath{A_{2}^{1}}}\\
	&+ 2\ip{\Cpo\iA\Cpo\iHmis\mpr, b}\tag{$A_{2}^{2}$}\\
	&+ \ip{\Cpo\iA\Cpo\Cpr^{-1}\mpr, \iA\Cpo\Cpr^{-1}\mpr}			      \tag{\ensuremath{A_{3}^{1}}}\\
	&+ 2\ip{\Cpo\iA\Cpo\Cpr^{-1}\mpr, b}\tag{$A_{3}^{2}$}\\
	&+ \ip{\Cpo b, b}													  \tag{\ensuremath{A_{3}^{3}}}.
\end{align}
Using the identity $\Cpo^{-1} = \iHmis + \Cpr^{-1}$, it follows that
$$A_{2}^{2} + A_{3}^{2} = 2\ip{\Cpo\iA\Cpo(\iHmis + \Cpr^{-1})\mpr, b} = 2\ip{\Cpo\iA\mpr, b}.$$
Similarly, splitting $A_{2}^{1}$ and using that $\Cpo$ is selfadjoint,
\begin{align*}
	A_{1} + A_{2}^{1} &+ A_{3}^{1} = \paren{A_{1} + \half A_{2}		^{1}} + \paren{\half A_{2}^{1} + A_{3}^{1}}\\
	&= \ip{\Cpo\iA\Cpo\iHmis\mpr, \iA\Cpo(\iHmis + \Cpr^{-1})\mpr}
	+ \ip{\Cpo\iA\Cpo(\iHmis + \Cpr^{-1})\mpr, 					\iA\Cpo\Cpr^{-1}\mpr}\\
	&= \ip{\Cpo\iA\mpr, \iA\mpr}.
\end{align*}
Finally, we finish the simplification of the product terms by combining previous calculations with the remaining term $A_{3}^{3}$,
\begin{equation}\label{eq:Asum}
\begin{alignedat}{1}
	A_{1} + A_{2} + A_{3} &= \paren{A_{2}^{2} + A_{3}^{2}} + \paren{A_{1} + A_{2}^{1} + A_{3}^{1}} + A_3^3\\
	&= 2\ip{\Cpo\iA\mpr, b} 
		+ \ip{\Cpo\iA\mpr, \iA\mpr} 
		+ \ip{\Cpo b, b}\\
	&= \|\iA \mpr + b\|_{\Cpo}^{2}.
\end{alignedat}
\end{equation}

Lastly, we turn our attention to the trace terms. Combining the first two trace terms and manipulating, 
\begin{align*}
	B_{1} + B_{2} &= \tr\paren{(\iHmis\Cpr+\mc{I})					\iHmis\paren{\Cpo\iA}^{2}\Cpo}\\
	&= \tr\paren{\Cpo(\iHmis + \Cpr^{-1})\Cpr\iHmis\paren{\Cpo\iA}		^{2}}\\
	&= \tr\paren{\Cpr\iHmis\paren{\Cpo\iA}^{2}}.
\end{align*}
Adding in the remaining term $B_{3}$, the sum of the trace terms is computed to be
\begin{equation}\label{eq:Bsum}
\begin{alignedat}{1}
	\paren{B_{1} + B_{2}} + B_{3} &= \tr\paren{\paren{\Cpr\iHmis 	+ \mc{I}-\half\mc{I}}\paren{\Cpo\iA}^{2}}\\
	&= \tr\paren{\Cpr\paren{\iHmis + \Cpr^{-1}}\paren{\Cpo\iA}		^{2} - \half\paren{\Cpo\iA}^{2}}\\
	&= \tr\paren{\Cpr\iA\Cpo\iA}-\half\tr\big((\Cpo\iA)		^{2}\big).
\end{alignedat}
\end{equation}

Summing \eqref{eq:Asum} and \eqref{eq:Bsum}, we obtain 
\[
\begin{aligned}
\Psi &= \paren{A_{1} + A_{2} + A_{3}} + \paren{B_{1} + B_{2} + B_{3}} \\
	 &= \|\iA \mpr + b\|_{\Cpo}^{2} + \tr\paren{\Cpr\iA\Cpo\iA}-\half\tr\big((\Cpo\iA)^{2}\big).
\end{aligned}
\]
Finally, note that~\eqref{eq:psi_full} follows from the definitions of $\iA$ 
and $b$, given in \eqref{eq:Abc}.
\qed
\end{proof}

\section{Proof of Proposition~\ref{prp:rPsi}}
\label{appdx:spectral}

\begin{proof}

Proving the proposition is equivalent to manipulating the trace terms in~\eqref{eq:rPsi}. Recall that $\st_r = \mat I - \sum_{i=1}^r \gamma_i (\vec v_i \otimes \vec v_i)$, with $\{(\gamma_i, \vec v_i)\}_{i=1}^r$, as in 
\eqref{eq:St_lowrank}. We then claim that
\begin{align}
	\tr\paren{\Gpr\ddq\Gpo\ddq} &\approx \tr(\ddqt^2) - \sum_{i=1}^r \gamma_i \| \ddqt \vec v_i \|^2,
	\label{eq:T1_approx}\\
	\tr\paren{\paren{\Gpo\ddq}^{2}} &\approx 	
	\tr(\ddqt^2) - 2\sum_{i=1}^r \gamma_i \| \ddqt \vec v_i \|^2 + \sum_{i,j=1}^r 
	\gamma_i \gamma_j \ip{\ddqt \vec v_i, \vec v_j}^2.
	\label{eq:T2_approx}
\end{align}
This result follows from repeated applications of the cyclic property of the trace. We show the 
proof of the second equality and omit the first one for brevity. Using the definition of $\st_r$,
\[
\begin{aligned}
	\tr\paren{\paren{\rGpo\ddq}^{2}} = \tr\paren{\st_r \ddqt \st_r\ddqt}
	&= \tr\paren{ 
		(\mat I - \sum_i^{r} \gamma_i \vec v_i \otimes \vec v_i)\ddqt(\mat I - \sum_j^{r} \gamma_j \vec v_j \otimes \vec v_j) \ddq 
		}	
		\\
	&= \tr\paren{ 
		\big(\ddqt^2 - \sum_i^{r} \gamma_i (\ddqt \vec v_i \otimes \ddqt \vec v_i)\big) \big(\mat I - \sum_j^{r} \gamma_j \vec v_j \otimes \vec v_j\big) 
		}	
		\\
	&= \tr\paren{\ddqt^2} - 2 \sum_{i=1}^r \gamma_i \tr\paren{ \ddqt \vec v_i \otimes \ddqt \vec v_i} 
		+ \sum_{i,j=1}^r \gamma_i \gamma_j \tr\paren{ (\ddqt \vec v_i \otimes \ddqt \vec v_i) (\vec v_j \otimes \vec v_j)}
		\\ 
	&= \tr\paren{\ddqt^2} - 2 \sum_{i=1}^r \gamma_i \|\ddqt \vec v_i\|^2 
		+ \sum_{i,j=1}^r \gamma_i \gamma_j  \ip{\ddqt \vec v_i, \vec v_j}^2.
\end{aligned}	
\]
Here, we have also used the facts $\tr( \vec u \otimes \vec v) = \ipM{\vec u, \vec v}$ and 
$\tr( (\vec s \otimes \vec t) (\vec u \otimes \vec v)) = \ipM{\vec s, \vec v}\ipM{\vec t, \vec u}$, for 
$\vec s, \vec t, \vec u$, and $\vec v$ in $\mR^{N}_{\mat M}$. Substituting \eqref{eq:T1_approx} and \eqref{eq:T2_approx} into \eqref{eq:rPsi}, we arrive at the desired representation of $\PsiSpec$.
\qed
\end{proof}

\section{Proof of Proposition~\ref{proposition:low_rank}}\label{sec:proof_of_prp_lowrank}
\begin{proof}
	We begin by proving (a). 
    Considering $\fPtw = (\mat{I} + \fFt^* \Ws \fFt)^{-1}$, we note that
    \[
       \mat{I} + \fFt^* \Ws \fFt = \mat{I} + \mat{M}^{-1} \fFt^\top\Ws\fFt 
       = \mat{M}^{-1}(\mat{M} + \fFtw^\top\fFtw).   
    \]
    Thus, $\fPtw = (\mat{M} + \fFtw^\top\fFtw)^{-1}\mat{M}$, 
    and the Sherman--Morrison--Woodbury identity provides that
    \[
        (\mat{M} + \fFtw^\top\fFtw)^{-1} = \mat{M}^{-1} - \mat{M}^{-1} \fFtw^\top 
        (\mat{I} + \fFtw \mat{M}^{-1} \fFtw^\top)^{-1}\fFtw\mat{M}^{-1}.
    \]
    Therefore, 
    \[
       \begin{aligned}
       (\mat{I} + \fFt^* \Ws \fFt)^{-1} 
       &= 
       \mat{I} - \mat{M}^{-1}\fFtw^\top(\mat{I} + \fFtw \mat{M}^{-1} \fFtw^\top)^{-1}\fFtw\\
       &= 
       \mat{I} - \fFtw^*(\mat{I} + \fFtw \fFtw^*)^{-1}\fFtw\\
       &=
       \mat{I} - \fFtw^*\fDw \fFtw.
       \end{aligned}
    \]
    Parts (b) and (c) follow from some algebraic manipulations and using the identity for $\fPtw$.
\qed
\end{proof}

\section{Gradient and Hessian of goal as in Section~\ref{sec:nonlinear_model_and_goal}}
\label{appdx:derivatives}
We obtain the adjoint-based expressions for the 
gradient and Hessian of $\Q$ following a formal Lagrange approach. 
This is accomplished by forming weak representations of
the inversion model \eqref{eq:model2} and prediction model \eqref{eq:prediction} 
and formulating  
a Lagrangian functional $\iL$ constraining
the goal to these forms.
In what
follows, we denote
\[
\begin{aligned}
\ms{V}^{p} 		&\defeq \crbra{p \in H^{1}(\Omega): \ p\big\vert_{E_{0}^{p}} = 0, p\big\vert_{E_{1}^{p}} = 1}, \\
\ms{V}_{0}^{p}  &\defeq \crbra{p \in H^{1}(\Omega): \ p\big\vert_{E_{0}^{p}\cup E_{1}^{p}}=0},\\
\ms{V}^{c} 		&\defeq \crbra{c \in H^{1}(\Omega): \ c\big\vert_{E_{0}^{c}\cup E_{1}^{c}}=0}.
\end{aligned}
\]
We next discuss the weak formulations of the inversion and prediction models.
The weak form of the inversion model is
\[
	\text{find} \ p \in \ms{V}^p \ \text{such that} \ \ip{\kap \grad p, \grad \lam} = \ip{m, \lam}, \quad \text{for all } \ \lam \in \ms{V}_{0}^{p}.
\]
Similarly, 
the weak formulation of the prediction model is
$$\text{find} \ c \in \ms{V}^{c} \ \text{such that} \ \alp\ip{\grad c, \grad\zet} + \ip{c\kap\grad p, \grad\zet} = \ip{f, \zet}, \ \text{for all} \ \zet \in \ms{V}^{c}.$$
We constrain the goal-functional to these weak forms, arriving at the Lagrangian
\begin{equation}\label{eq:L1}
\iL(c, p, m, \lam, \zet) = \ip{\ind, c} + \ip{\kap \grad p, \grad \lam} - \ip{m, \lam} + \alp\ip{\grad c, \grad\zet} + \ip{c\kap\grad p, \grad\zet} - \ip{f, \zet}.
\end{equation}
Here, $\lam$ and $\zet$ are the Lagrange multipliers. This Lagrangian 
facilitates computing the derivative of the goal functional with 
respect to the inversion parameter $m$. 

\boldheading{Gradient}
The gradient expression is derived using the formal Lagrange 
approach~\cite{Troltzsch10}. 
Namely, the Gâteaux derivative of $\Q$ at $m$, and in a direction $\tilde{m}$, 
satisfies 
\begin{equation}\label{eq:gradient}
\iL_{m}[\tilde{m}] = -\ip{\lam, \tilde{m}} \equiv \ip{\grad\Q(m), \tilde{m}},
\end{equation}
provided the variations of the Lagrangian with respect to the remaining arguments 
vanish. Here $\iL_{m}[\tilde{m}]$ is shorthand for  
\[
\iL_{m}[\tilde{m}] := \frac{d}{d\eta}\bigg|_{\eta=0}
\iL(c, p, m + \eta \,\tilde{m}, \lam, \zet).
\]
Thus, an evaluation of the gradient requires solving the following system,
\begin{equation}\label{eq:system1}
\iL_{\lam}[{\tilde{\lam}}] = 0, \ \iL_{\zet}[\tilde{\zet}] = 0, \ \iL_{c}[\tilde{c}] = 0, \ \text{and} \ \iL_{p}[\tilde{p}] = 0,
\end{equation}
along all test functions $\tilde{\lam}, \tilde{\zet}, \tilde{c}, \tilde{p}$ 
in the respective test function spaces. 
The equations are solved in the order presented in \eqref{eq:system1}. 
It can be shown that the weak form of the inversion model is equivalent to
$\iL_{\lam}[{\tilde{\lam}}] = 0$, similarly the prediction model weak form is
equivalent to $\iL_{\zet}[\tilde{\zet}]=0$. These are referred to as the
\emph{state equations}. We refer to $\iL_{c}[\tilde{c}]=0$ and
$\iL_{p}[\tilde{p}]=0$ as the \emph{adjoint equations}. The variations required
to form the gradient system are
\begin{equation*}
\begin{alignedat}{1}
\iL_{\lam}[\tilde{\lam}] &= \ip{\kap \grad p, \grad \tilde{\lam}} - \ip{m, \tilde{\lam}},\\
\iL_{\zet}[\tilde{\zet}] &= \alp\ip{\grad c, \grad\tilde{\zet}} + \ip{c\kap\grad p, \grad\tilde{\zet}} - \ip{f, \tilde{\zet}},\\
\end{alignedat}
\quad
\begin{alignedat}{1}
\iL_{c}[\tilde{c}] &= \ip{\ind, \tilde{c}} + \alp\ip{\grad\zet,\grad\tilde{c}} + \ip{\kap\grad\zet\cdot\grad p, \tilde{c}},\\
\iL_{p}[\tilde{p}] &= \ip{\kap\grad\lam, \grad\tilde{p}} + \ip{\kap c\grad\zet, \grad\tilde{p}}.
\end{alignedat}
\end{equation*}

\boldheading{Hessian} 
We compute the action of the Hessian using a formal Lagrange approach as well. 
This is facilitated by formulating a \emph{meta-Lagrangian} functional; 
for a discussion of this approach, see, e.g.,~\cite{VillaPetraGhattas21}.
The meta-Lagrangian is
\begin{equation}\label{eq:L2}
\begin{alignedat}{1}
\iL^{H}(c, p, m, \lam, \zet, \hat{c}, \hat{p}, \hat{\lam}, \hat{\zet}, \hat{m}) &= -\ip{\lam,\hat{m}}\\
&+ \ip{\kap\grad p, \grad\hat{\lam}}-\ip{m,\hat{\lam}}\\
&+ \alp\ip{\grad c, \grad\hat{\zet}} + \ip{c\kap\grad p, \grad\hat{\zet}} - \ip{f, \hat{\zet}}\\
&+ \ip{\ind,\hat{c}} + \alp\ip{\grad\zet,\grad\hat{c}}+\ip{\kap\grad\zet\cdot\grad p,\hat{c}}\\
&+ \ip{\kap\grad\lam, \grad\hat{p}} + \ip{\kap c\grad\zet,\grad\hat{p}},
\end{alignedat}
\end{equation}
where 
$\hat{p} \in \ms{V}_{0}^{p}, \hat{\lam} \in \ms{V}_{0}^{p}, 
\hat{c} \in \ms{V}^{c}$, and
$\hat{\zet} \in \ms{V}^{c}$ are additional Lagrange multipliers. 
Equating variations of $\iL^{H}$ with respect to $\hat{c}, \ \hat{p}, \ \hat{\lam}, \ \hat{\zet}$, and $\hat{m}$ to zero returns 
the gradient system. To apply the Hessian of $\Q$ at $m \in \ms{M}$ to $\hat{m} \in \ms{M}$ (in the $\tilde{m}$ direction), 
we must solve the gradient system \eqref{eq:system1}, then the additional system
\begin{equation}\label{eq:system2}
\begin{alignedat}{1}
\iL^{H}_{\lam}[{\tilde{\lam}}] = 0,\ \iL^{H}_{\zet}[\tilde{\zet}] = 0,\ \iL^{H}_{c}[\tilde{c}] = 0,\ \iL^{H}_{p}[\tilde{p}] &= 0,
\end{alignedat}
\end{equation}
for all test functions $\tilde{\lam}$, $\tilde{\zet}$, $\tilde{c}$, and $\tilde{p}$. 
The equations $\iL^{H}_{\lam}[{\tilde{\lam}}] = 0$ and $\iL^{H}_{\zet}[\tilde{\zet}] = 0$ are referred to as 
the \emph{incremental state} equations. We call the equations 
$\iL^{H}_{c}[\tilde{c}] = 0$ and $\iL^{H}_{p}[\tilde{p}] = 0$ the \emph{incremental adjoint} equations. 
For readers' convenience, we provide the required variational derivative for 
forming the incremental equations:
\begin{equation*}
\begin{alignedat}{1}
\iL^{H}_{\lam}[\tilde{\lam}] &= -\ip{\hat{m}, \tilde{\lam}} + \ip{\kap\grad\hat{p}, \tilde{\lam}},\\
\iL^{H}_{\zet}[\tilde{\zet}] &= \alp\ip{\grad\hat{c},\grad\tilde{\zet}} + \ip{\hat{c}\kap\grad p, \grad\tilde{\zet}} + \ip{c\kap\grad\hat{p}, \grad\tilde{\zet}},\\
\end{alignedat}
\quad
\begin{alignedat}{1}
\iL^{H}_{c}[\tilde{c}] &= \alp\ip{\grad\hat{\zet},\grad\tilde{c}} + \ip{\kap \grad\hat{\zet} \cdot \grad p, \tilde{c}} + \ip{\kap\grad \zet \cdot \grad \hat{p}, \tilde{c}},\\
\iL^{H}_{p}[\tilde{p}] &= \ip{\kap\grad\hat{\lam},\grad\tilde{p}} + \ip{c\kap\grad \hat{\zet}, \grad \tilde{p}} + \ip{\hat{c}\kap\grad \zet, \grad\tilde{p}}.
\end{alignedat}
\end{equation*}
The variation of $\iL^{H}$ with respect to $m$ reveals a means to 
compute the Hessian vector product, $\grad^{2}\Q(m)\hat{m}$, as follows.
\begin{equation}\label{eq:hessian}
\iL^{H}_{m}[\tilde{m}] = -\ip{\hat{\lam},\tilde{m}} = \ip{\grad^{2}\Q(m)\hat{m}, \tilde{m}}.
\end{equation}

\end{document}

%% file: macros.tex
\usepackage{hyperref}
\usepackage[margin=1in]{geometry}
\usepackage[utf8]{inputenc}
\usepackage{amssymb,amsmath,amsfonts,verbatim,mathtools,graphicx}
\usepackage{booktabs}
\usepackage{esint}
\usepackage{xcolor}
\usepackage{float}
\usepackage[section]{placeins}
\usepackage{listings}
\usepackage{mathrsfs}
\usepackage{algorithm}
\usepackage{algorithmic}
\usepackage{float}
\usepackage{upgreek}
\usepackage{dsfont}
\usepackage[shortlabels]{enumitem}

\usepackage{diagbox}
\newcommand{\mc}[1]{\mathcal{#1}}
\newcommand{\mb}[1]{\mathbb{#1}}
\newcommand{\ms}[1]{\mathscr{#1}}
\newcommand{\mf}[1]{\mathsf{#1}}
\renewcommand{\vec}[1]{\boldsymbol{#1}}
\newcommand{\mat}[1]{\mathbf{#1}}

\newcommand{\pr}{\text{pr}}
\newcommand{\po}{\text{post}}
\newcommand{\mis}{\text{mis}}
\newcommand{\MAP}{\text{MAP}}

\newcommand{\qua}{\text{quad}}

\newcommand{\iC}{\mc{C}}

\newcommand{\iA}{\mc{A}}
\newcommand{\iL}{\mc{L}}
\newcommand{\iF}{\mc{F}}
\newcommand{\iG}{\mc{G}}
\newcommand{\iS}{\mc{S}}
\newcommand{\iB}{\mc{B}}
\newcommand{\iH}{\mc{H}}

\newcommand{\fF}{\mat{F}}

\newcommand{\fH}{\mat{H}}
\newcommand{\Cpo}{\mc{C}_{\po}}
\newcommand{\Cpr}{\mc{C}_{\pr}}

\newcommand{\Gpo}{\mat{\Gamma}_{\po}}
\newcommand{\Gpr}{\mat{\Gamma}_{\pr}}
\newcommand{\rGpo}{\mat \Gamma_{\po,\text{k}}}

\newcommand{\Q}{\ensuremath{\mathcal{Z}}}
\newcommand{\dQ}{\bar{g}_{\scriptscriptstyle{\Q}}}
\newcommand{\ddQ}{\bar{\mc{H}}_{\scriptscriptstyle{\Q}}}

\newcommand{\q}{Z}
\newcommand{\dq}{\bar{\vec g}_{\text{z}}}
\newcommand{\ddq}{\bar{\mat H}_{\text{z}}}
\newcommand{\ddqt}{\tilde{\mat H}_{\text{z}}}
\newcommand{\bq}{\bar{\vec b}_{\text{z}}}

\newcommand{\fFt}{\tilde{\mat{F}}}
\newcommand{\fDw}{\mat{D}_{\vec{w}}}
\newcommand{\fFtw}{\tilde{\mat{F}}_{\vec{w}}}
\newcommand{\fFtwr}{\tilde{\mat{F}}^{r}_{\vec{w}}}
\newcommand{\fPtw}{\tilde{\mat{P}}_{\vec{w}}}
\newcommand{\Ws}{\mat{W}_{\!\sigma}}

\newcommand{\fmpr}{\vec{m}_\text{pr}}
\newcommand{\fmbar}{\bar{\vec{m}}}

\newcommand{\grad}{\nabla}
\newcommand{\Del}{\Delta}
\newcommand{\half}{\frac{1}{2}}
\newcommand{\tr}{\mathrm{tr}}
\newcommand{\limn}[1]{\lim_{#1\rightarrow\infty}}
\newcommand{\defeq}{\vcentcolon=}
\newcommand{\R}{\mathbb{R}}

\newcommand{\iHmist}{\tilde{\iH}_{\mis}}

\newcommand{\st}{\tilde{\mat P}}
\newcommand{\rst}{\tilde{\mat P}_\text{k}}
\newcommand{\iHmis}{\iH_{\mis}}

\newcommand{\fHmist}{\tilde{\fH}_{\mis}}

\newcommand{\mE}{\mb{E}}
\newcommand{\mV}{\mb{V}}
\newcommand{\mR}{\mb{R}}
\newcommand{\mN}{\mb{N}}

\newcommand{\mpr}{m_{\pr}}

\newcommand{\mMAPy}{m_{\MAP}^{\vec{y}}}
\newcommand{\mbar}{\bar{m}}
\newcommand{\mtrue}{m_{\text{true}}}

\newcommand{\mMAPyw}{m_{\MAP}^{\vec{y}, \vec w}}

\newcommand{\mapA}{\mc{K}}
\newcommand{\mapB}{k}

\newcommand{\sig}{\sigma}
\newcommand{\kap}{\kappa}
\newcommand{\alp}{\alpha}

\newcommand{\zet}{\zeta}
\newcommand{\lam}{\lambda}
\newcommand{\gam}{\gamma}

\newcommand{\paren}[1]{\left( #1 \right)}
\newcommand{\sqbra}[1]{\left[ #1 \right]}
\newcommand{\crbra}[1]{\left\{ #1 \right\}}
\newcommand{\ip}[1]{\left\langle #1 \right\rangle}
\newcommand{\ipM}[1]{\left\langle #1 \right\rangle_{\mat M}}

\newcommand{\PsiRand}{\mat{\Psi}_{\text{rand,p}}}
\newcommand{\PsiSpec}{\mat{\Psi}_{\text{spec,k}}}
\newcommand{\PsiSVD}{\mat{\Psi}_{\text{svd,r}}}

\newcommand{\gl}{G_{\ell}}
\newcommand{\gq}{G_{q}}

\newcommand{\postm}{\mu_{\po}^{\vec{y}}}
\newcommand{\boldheading}[1]{\par\textbf{{#1}.}}

\newcommand{\ind}{\mathds{1}_{\Omega^{*}}}
\renewcommand{\qed}{\hfill$\square$}
\newcommand{\summ}[4]{\sum_{{#1},{#2},{#3},{#4}=1}^n} 
\newcommand{\Ae}[2]{\ip{\iA e_{{#1}}, e_{{#2}}}}